\documentclass{lms}

\usepackage[utf8]{inputenc} 

\usepackage{amsmath}
\usepackage{amssymb}

\usepackage[all]{xy}

\usepackage{stmaryrd}

\usepackage{slashed}

\newcommand{\C}{{\mathbb{C}}}
\newcommand{\R}{{\mathbb{R}}}
\newcommand{\Z}{{\mathbb{Z}}}
\newcommand{\N}{{\mathbb{N}}}
\newcommand{\K}{{\mathfrak{K}}}
\newcommand{\B}{{\mathfrak{B}}}

\newcommand{\cE}{{\mathcal{E}}}
\newcommand{\cfrak}{{\mathfrak{c}}}
\newcommand{\cbar}{{\overline{\mathfrak{c}}}}
\newcommand{\kfrak}{{\mathfrak{uc}}}
\newcommand{\kbar}{{\overline{\mathfrak{uc}}}}
\newcommand{\cF}{{\mathcal{F}}}
\newcommand{\cO}{{\mathcal{O}}}
\newcommand{\MA}{{\mathfrak{M}}}

\newcommand{\FA}{{C_r^*(M,\cF)}}
\newcommand{\FC}{{\cO(M,\cF)}}
\newcommand{\fol}{{(M,\cF)}}

\newcommand{\hatotimes}{\mathbin{\widehat\otimes}}

\DeclareMathOperator{\id}{id}
\DeclareMathOperator{\im}{im}

\DeclareMathOperator{\ind}{ind}
\DeclareMathOperator{\supp}{supp}
\DeclareMathOperator{\codim}{codim}

\newtheorem{thm}{Theorem}[section]
\newtheorem{lem}[thm]{Lemma}

\newtheorem{question}[thm]{Question}
\newtheorem{prop}[thm]{Proposition}
\newtheorem{cor}[thm]{Corollary}

\newnumbered{example}[thm]{Example}
\newnumbered{mdef}[thm]{Definition}

\newnumbered{acknow}{Acknowledgements}

\newnumbered{rem}[thm]{Remark}
\newnumbered{qstn}[thm]{Question}

\title[$K$-theory of leaf spaces]{Ring and module structures on  $K$-theory of leaf spaces and their application to longitudinal index theory}
\author{Christopher Wulff}
\extraline{Paper accepted for publication by the Journal of Topology.}

\classno{19K56 (primary), 46L80, 46L87, 57R30, 58J22  (secondary)}

\begin{document}
\maketitle
\begin{abstract}
Pursuing conjectures of John Roe, we use the stable Higson corona of foliated cones to construct a new $K$-theory model for the leaf space of a foliation.
This new $K$-theory model is -- in contrast to Alain Connes' $K$-theory model -- a ring.
We show that Connes' $K$-theory model is a module over this ring and develop an interpretation of the module multiplication in terms of indices of twisted longitudinally elliptic operators.
\end{abstract}

%\keywords{$K$-theory of leaf spaces, foliated cones, stable Higson corona, longitudinal index theory, index formulas.}

\section{Introduction}
Let $\fol$ be a foliation on a compact manifold $M$. 
Roe suggested in \cite{RoeFoliations} to use the coarse geometry of the foliated cone $\FC$ (ibid.) to construct a new $K$-theory model $K_{FJ}^*(M/\mathcal{F})$ for the leaf space. 
The index $FJ$ stands for Farrell-Jones, as this construction was motivated by the foliated control theory of \cite{FarrellJones}.

However, it is virtually impossible to work with Roe's definition, because it involves using $K$-homology of the Roe-algebra. The Roe-algebra is non-separable and non-nuclear and therefore the usual tools for calculating $K$-homology fail, not to mention the fact that it is not even clear how to define $K$-homology in this case correctly.

The purpose of the present paper is to present a slight modification of Roe's $K$-theory model which is well behaved and for which all conjectures from \cite{RoeFoliations} are true. 
It is based on the idea of 
\cite{EmeMey} that one should use the $K$-theory of the \emph{stable Higson corona} $\cfrak(X)$ of a coarse space $X$ as a replacement for the $K$-homology of the Roe algebra.
With this in mind we propose the following alternative definition:
\begin{mdef*}[{\ref{def:FJModel}}\footnote{Only part of Definition \ref{def:FJModel} is displayed here in the introduction. For the full definition, see Section \ref{sec:newKtheory}.}]
The ``Farrell-Jones'' model for the $K$-theory of the leaf space of a foliation $\fol$ is
$K^{*}_{FJ}(M/\mathcal{F}):=K_{-*}(\cfrak(\FC))$. 
\end{mdef*}

Roe conjectured that the Farrell-Jones $K$-theory groups $K^{*}_{FJ}(M/\mathcal{F})$ are in fact rings. Using our modified definition, the ring structure is canonical and was introduced in \cite[Definition 6.1]{WulffCoassemblyRinghomo}.

One might expect  further basic properties of  $K$-theory of ``spaces'', and indeed the new model satisfies the following:
The  rings $K^*_{FJ}(M/\mathcal{F})$ are contravariantly functorial under smooth maps of leaf spaces (cf.\ Theorem \ref{thm:Ktheorymodelfunctor}) and if $\fol$ comes from a fibre bundle $M\to B$ with connected fibre,
then there is a canonical ring isomorphism
$K^*_{FJ}(M/\mathcal{F})\cong K^*(B)$ (Example \ref{ex:submersion}).

Usually, one considers the $K$-theory of Connes' foliation algebra, 
\[K_C^*(M/\mathcal{F}):=K_{-*}(\FA),\]
as the $K$-theory of the leaf space \cite[Sections 5,6]{ConSurvey}.
We use the reduced foliation algebra $\FA$, because our proofs of the main theorems don't work for the full foliation algebra $C^*(M,\cF)$. These groups are the right receptacles of indices of longitudinally elliptic operators (see \cite[Section 7]{ConSurvey} and Section \ref{sec:longindtheory}) and there is a wrong way functoriality $f_!:K_C^*(M_1/\cF_1)\to K_C^*(M_2/\cF_2)$ under $K$-oriented smooth maps of leaf spaces $f:M_1/\cF_1\to M_2/\cF_2$ \cite{HilSka}.

Roe asked in \cite{RoeFoliations} for the relation between Connes' $K$-theory model and the new Farrell-Jones $K$-theory model.
With our modified definition of the latter, Roe's conjectures work out fine:
\begin{cor*}[{\ref{cor:modulemultiplication}\footnote{Ditto: See Section \ref{sec:modulestructure} for the complete Corollary \ref{cor:modulemultiplication}.} (cf.\ \cite[Conjecture 0.2]{RoeFoliations})}]
 $K^*_C(M/\cF)$ is a module over $K_{FJ}^*(M/\cF)$.
\end{cor*}
\begin{cor*}[{\ref{cor:shriekmap} (cf.\ \cite[p.\ 204]{RoeFoliations})}]
Assume that $T\mathcal{F}$ is even dimensional and $\operatorname{spin}^c$ and let $\slashed{D}$ be the corresponding Dirac operator.
Then the map
\[p_!\circ p^*:K^*_{FJ}(M/\cF)\to K^*_C(M/\cF)\]
is module multiplication with $\ind(\slashed{D})\in K_C^0(M/\cF)$.
\end{cor*}
Here, $p:M\to M/\cF$ is the canonical smooth map of leaf spaces, its domain being $M/\cF_0$ for the trivial $0$-dimensional foliation $\cF_0$ on $M$ (cf.\ Example \ref{ex:projtoleafspace}).

Proving the same result for odd dimensional $\operatorname{spin}^c$ foliations would require the index theory of selfadjoint longitudinally elliptic operators and its relation to the module structure, which are not discussed in this paper.

The module structure can also be interpretated by indices of twisted longitudinally elliptic operators. If $D$ is a longitudinally elliptic operator on $\fol$ and $F\to M$ a smooth vector bundle, then the twisted operator $D_F$ again is longitudinally elliptic. 
There is no general method to calculate
$\ind(D_F)$ from $\ind(D)$ and $F$ alone, because there are cases where $\ind(D)$ vanishes but $\ind(D_F)$ does not.
It is possible, however, if the bundle $F$ is a bundle over the leaf space in an asymptotic sense, as defined in Definition  \ref{asymptoticallyleafwisebundles}. This condition ensures that $[F]\in K^0(M)$ is the pullback of an element $x_F\in K^0_{FJ}(M/\cF)$.
\begin{cor*}[\ref{cor:twistingbyasymptoticallyleafwisebundles}]
If $D$ is a longitudinally elliptic  operator, $F\to M$ a smooth vector bundle for which there is an element $x_F\in K^0_{FJ}(M/\mathcal{F})$ with $[F]=p^*(x_F)$, 
then the index of the twisted operator $D_F$ is
\[\ind(D_F)=x_F\cdot \ind(D)\in K^0_C(M/\mathcal{F}).\]
\end{cor*}
An illustrative special case is provided by  fibre bundles $p:M\to B$ with connected fibre:
Here, a longitudinally elliptic differential operator $D$ is a family of operators parametrized by $B$ and the pullback $F=p^*F'$ of a vector bundle $F'$ over $B$ is asymptotically a bundle over the leaf space.
Under the canonical isomorphism $K^*_C(M/\cF)\cong K^*(B)$ \cite[Section 5]{ConSurvey}, the indices  $\operatorname{ind}(D), \operatorname{ind}(D_F)$ correspond to the family indices of $D,D_F$, and under the isomorphism $K^*_{FJ}(M/\cF)\cong K^*(B)$ mentioned above, $x_F$ corresponds to $[F']\in K^0(B)$. 
In this case, the corollary specializes to the rather obvious statement 
\[\operatorname{ind}(D_F)=[F']\cdot\operatorname{ind}(D)\in K^0(B)\,.\]

Finally, it should be said that Section \ref{sec:ConesFunctorial} of this paper contains the proof of functoriality of the foliated cones construction, which is simultaneously a topic of autonomous interest:
\begin{thm*}[\ref{thm:conefunctoriality}]
The foliated cone construction is a functor from the category of foliations and smooth maps of leaf spaces between them to the category of metrizable coarse spaces and coarse equivalence classes of coarse Borel maps between them.
\end{thm*}
For the readers convenience we note that the full power of functoriality is not needed for understanding the main results of this paper. Only the special cases of Examples \ref{ex:projtoleafspace}, \ref{example:submersionsmoothmaps}, where the induced maps have a very simple description, are relevant later on. Thus, the impatient reader may simply skip the first part of  Section \ref{sec:ConesFunctorial}.

One last comment: in this paper we shall use the symbol $\otimes$ for the \emph{maximal} tensor product of $C^*$-algebras.

\begin{acknow*}
This paper arises from the author's doctoral thesis at the University of Augsburg.
The author would like to thank his thesis advisor Bernhard Hanke for his steady encouragement and advice. Furthermore, the author is grateful to Thomas Schick and the anonymous referee for helpful comments.
Part of this work was carried out during a research stay at the Max Planck Institute for Mathematics in Bonn, whose hospitality is gratefully acknowledged.
The doctoral project was supported by a grant of the Studienstiftung des Deutschen Volkes.
\end{acknow*}

\tableofcontents

\section{Basic definitions}\label{sec:basicdef}
We begin by recalling some basic notions which are relevant throughout this paper. These are: the foliated cone and the stable Higson corona, which are the ingredients of the new $K$-theory model; the holonomy groupoid, which appears in the definition of Connes $K$-theory model and in the definition of the notion of smooth maps between leaf spaces. 
Furthermore, we introduce the length function, which is the main tool for relating the holonomy groupoid to coarse geometry of foliated cones, thus providing the link between the two $K$-theory models.

\subsection{Foliated cones}\label{subsec:foliatedcones}
Let $(M,\cF)$ be a foliation of a compact manifold $M$. 
Choose any Riemannian metric $g_M$  on $M$ (the foliation does \emph{not} have to be a Riemannian foliation),
and let $g_N$ be its normal component with respect to the orthogonal decomposition $TM=T\cF\oplus N\cF$ of the tangent bundle into longitudinal and normal vectors.
\begin{mdef}[({cf.\ \cite{RoeFoliations}})]
The \emph{foliated cone} $\cO(M,\cF)$ over $\fol$ is the manifold $[0,\infty)\times M$ equipped with the Riemannian metric 
\[g:=dt^2+g_M+t^2g_N\]
which blows up only in the transverse direction as the coordinate $t\in[0,\infty)$ tends to infinity. 
For the  the trivial  $0$-dimensional foliation $\cF_0$ on $M$ we abbreviate  $\cO M:=\cO(M,\cF_0)$.
\end{mdef}

From a topological point of view one would have to crush $\{0\}\times M$ to a point to obtain a cone (which Roe did in his original definition), but as we are only interested in the coarse geometry, we may neglect this difference.
Furthermore, the coarse equivalence class of the foliated cone is also independent from the chosen metric $g_M$.
This is easily seen later on as a byproduct of our exposition.
Note also that $\cO M$ is coarsely equivalent to the usual Riemannian cone over $M$.

The idea behind the foliated cone is that from a coarse geometric perspective one sees the leaves diverging from each other as $t\to\infty$ while longitudinal distances stay bounded and thus become irrelevant.
Therefore, coronas of this coarse space may be thought of as models for the leaf space of a foliation.
In the present paper we implement this idea by using the \emph{stable Higson corona}:

\subsection{The stable Higson corona}
The stable Higson corona of a coarse space was introduced in \cite{EmeMey}. 
Its definition and the properties which are relevant to us are also summarized in \cite[Sections 1 \& 6]{WulffCoassemblyRinghomo}.
We briefly recall what is needed:
\begin{mdef}[({cf.\ \cite[Section 3]{EmeMey},\cite[Definition 1.4]{WulffCoassemblyRinghomo}})]
\label{coarsefunctionalgebras}
Let $(X,d_X)$ be a proper metric space, $(Y,d_Y)$ a metric space and $D$ any $C^*$-algebra.
\begin{enumerate}
\item A Borel map $f:X\to Y$
is said to have \emph{vanishing variation}, if for all $R>0$ the function
\[\operatorname{Var}_Rf:\,X\to [0,\infty),\quad x\mapsto\sup\{d_Y(f(x),f(y))\,|\,d_X(x,y)\leq R\}\]
vanishes at infinity. \cite[Definition 3.1]{EmeMey}
\item Let $\kbar(X;D)$ be the $C^*$-algebra of bounded, continuous functions of vanishing variation $X\to D$. It is called the \emph{unstable Higson compactification of $X$ with coefficients $D$}.
\item The \emph{unstable Higson corona of $X$ with coefficients $D$} is the quotient $C^*$-algebra $\kfrak(X;D):=\kbar(X;D)/C_0(X;D)$.
\item  Denote by $\K$ the $C^*$-algebra of compact operators on a fixed, infinite dimensional, separable Hilbert space $\ell^2$, e.\,g.\ $\ell^2(\N)$ or $\ell^2(\Z)$.
The \emph{stable} counterparts of the above function algebras are obtained by replacing $D$ with $D\otimes\K$:
\begin{align*}
\cbar(X;D)&:=\kbar(X;D\otimes\K),
\\\cfrak(X;D)&:=\kfrak(X;D\otimes\K).
\end{align*}
In particular, $\cfrak(X;D)$ is the \emph{stable Higson corona of $X$ with coefficients $D$}.  \cite[Definition 3.2]{EmeMey}
\item If $D=\C$, we usually omit $D$ from notation.
\end{enumerate}
\end{mdef}
\begin{prop}[({cf.\ \cite[Proposition 3.7]{EmeMey}, \cite[Proposition 1.5]{WulffCoassemblyRinghomo}})]\label{prop:coronasfunctorial}
The assignments $X\mapsto \kfrak(X,D),\, X\mapsto \cfrak(X,D),\, X\mapsto \cfrak(X),$ are contravariant functors from the category of proper metric spaces and coarse equivalence classes of coarse Borel maps to the category of $C^*$-algebras.

The assignments $X\mapsto \kbar(X,D),\, X\mapsto \cbar(X,D),\, X\mapsto \cbar(X),$ are contravariantly functorial with respect to continuous coarse maps.

Furthermore, there is the obvious covariant functoriality in the coefficient algebra.
\end{prop}

The main theme of \cite{WulffCoassemblyRinghomo} is the canonical multiplicative structure on the $K$-theory of the corona algebras:
By pointwise multiplication of functions we obtain a  $*$-homomorphism
\[\nabla:\,\kfrak(X;D)\otimes \kfrak(X;E)\to\kfrak(X; D\otimes  E)\,.\]
\begin{mdef}[({\cite[Definition 6.1]{WulffCoassemblyRinghomo}})]\label{def:coronaproduct}
The product
\[K_i(\kfrak({X};D))\otimes K_j(\kfrak({X};E))\to K_{i+j}(\kfrak({X};D\otimes E))\]
is the composition of the exterior product in $K$-theory with  $\nabla_*$. Replacing $D,E$ by $D\otimes \K,E\otimes \K$ or simply both by $\K$ we obtain the products
\begin{align*}
K_i(\cfrak({X};D))\otimes K_j(\cfrak({X};E))&\to K_{i+j}(\cfrak({X};D\otimes E))
\\K_i(\cfrak({X}))\otimes K_j(\cfrak({X}))&\to K_{i+j}(\cfrak({X}))
\end{align*}
for the stable Higson corona. We denote all of them simply by ``$\,\cdot\,$''.
\end{mdef}
These products are associative, graded commutative and independent of the choice of the identification $\K\otimes\K\cong \K$ which is hidden in the definition.
In particular, $K_*(\cfrak(X))$ is canonically a $\Z_2$-graded, graded commutative ring.

\subsection{The holonomy groupoid and its length function}
The \emph{holonomy groupoid} (also called the \emph{graph}) of a foliation $\fol$ was introduced in  \cite{Winkelnkemper}. In this section, we briefly review its definition and relate it to the coarse geometry of foliated cones via its length function.

As always, let $(M,\cF)$ be a foliation of a compact manifold $M$.
Let $c$ be a leafwise path (i.\,e.\ $c$ is piecewise smooth and $\dot c(t)\in T_{c(t)}\cF$ for all $t$) between $x_0,x_1\in M$.
Choose foliation charts 
\[\phi_i:\,U_i\xrightarrow{\approx} V_i\times W_i\subset \R^{\dim \cF}\times \R^{\codim\cF}\]
around $x_i,\,i=0,1$. By following paths which stay close to $c$, we obtain a well defined germ of local diffeomorphisms $W_0\to W_1$ at $\phi_0(x_0)$. This germ is called the holonomy of $c$ (with respect to the chosen coordinate charts).

Two leafwise paths $c_1,c_2$ between $x_0,x_1$ are said to have the same holo\-nomy if their holonomy with respect to some foliation charts around $x_0,x_1$ agree. This notion is independent of the choice of the charts.
\begin{mdef}[({\cite{Winkelnkemper}})]
The \emph{holonomy groupoid} (also called \emph{graph}) \mbox{$G\overset{s}{\underset{r}{\rightrightarrows}}M$} of $(M,\cF)$ is the smooth (possibly non-Hausdorff) groupoid of dimension $\dim M+\dim\cF$ consisting of holonomy classes of leafwise paths together with the obvious source and range maps.
\end{mdef}
The manifold structure on $G$ is obtained as follows:  For a leafwise path $c$, choose $\phi_0,\phi_1$ as above. We may assume that the chart domains are small enough such that $V_0,V_1,W_0,W_1$ are homeomorphic to open balls and $W_0\cong W_1$ is a diffeomorphism representing the holonomy of $c$. Then there is an obvious chart from a set of holonomy classes of leafwise paths which stay close to $c$ onto the  open subset $V_0\times W_0\times V_1\subset\R^{\dim M+\dim\cF}$.

For each $x\in M$, 
the sets $G_x:=s^{-1}(\{x\}),G^x:=r^{-1}(\{x\})$ are smooth Hausdorff submanifolds of $G$ of dimension $\dim\mathcal{F}$.
For arbitrary subsets $A,B\subset M$ we define
\[G_A:=s^{-1}(A),\quad G^B:=r^{-1}(B),\quad G_A^B:=s^{-1}(A)\cap r^{-1}(B).\]

The tool which relates the holonomy groupoid to coarse geometry is the length function: Assume that $M$ is equipped with a Riemannian metric $g_M$.
\begin{mdef}
The \emph{length function} of $G$ assigns to each holonomy class the infimum of the lengths of its representatives:
\[L:G\to [0,\infty),\quad \gamma\mapsto \inf_{c\in\gamma}L(c)\]
\end{mdef}
\begin{lem}
The length function is upper semi-continuous.
\end{lem}
\begin{proof}
Let $a>0$. Assume $\gamma\in L^{-1}([0,a))$ is represented by a piecewise smooth path $c$ of length $L(c)<a$. From the construction of the manifold structure on $G$ it is evident that there is an open neighbourhood $U$ of $\gamma$ in (some coordinate chart of) $G$ whose elements can be represented by a family of piecewise smooth paths $\{c_\rho\}_{\rho\in U}$ in such a way that $U\to [0,\infty),\,\rho\mapsto L(c_\rho)$ is continuous and $c_\gamma=c$. Thus, $U_0:=\{\rho\in U:\,L(c_\rho)<a\}$ is an open neighbourhood of $\gamma$. For any $\rho\in U_0$ we have $L(\rho)\leq L(c_\rho)<a$, so $U_0\subset L^{-1}([0,a))$.

This shows that $L^{-1}([0,a))\subset G$ is open and therefore $L:G\to [0,\infty)$ is upper semi-continuous.
\end{proof}
\begin{cor}\label{length:boundedoncompact}
The length function is bounded on every compact subset of $G$. 
\end{cor}

The length function now gives us the following relation between the holonomy groupoid and the coarse geometry of the foliated cone:
\begin{lem}
Let $K\subset G$ be compact. Then the set
\[E_K:=\{(t,s(\gamma),t,r(\gamma))\in \FC\times \FC:\,t\in [0,\infty),\,\gamma\in K\}\]
is contained in  
\[E_R:=\{(x,y)\in \FC\times \FC:\,\operatorname{dist}(x,y)\leq R\}\]
for some $R>0$.
\end{lem}
\begin{proof}
Let $R$ be a little bit bigger then the upper bound for $L|_K$ given by Corollary \ref{length:boundedoncompact}.
The points $(t,s(\gamma))$ and $(t,r(\gamma))$ in $\FC$ are connected by the path $c_t:\,s\mapsto(t,c(s))$, where $c$ is a representative of $\gamma$ of length $L(c)<R$. As the metric on $\FC$ blows up only in transverse direction, we have 
\[d\big((t,s(\gamma)),\,(t,r(\gamma))\big)\leq L(c_t)= L(c)<R\]
and the claim follows.
\end{proof}
\begin{cor}\label{commutationlimit}
Let $g\in\kbar(\FC,D)$ and denote its restriction to $\{t\}\times M\subset\FC$ by $g_t\in C(M)\otimes D$. Then the norm of $s^*g_t-r^*g_t\in C(G,D)$ restricted to any compact set $K$ tends to zero as $t$ goes to infinity.
\end{cor}
\begin{proof}
This follows directly from the previous Lemma and vanishing variation of $g$.
\end{proof}

\section{Functoriality of the foliated cone construction}\label{sec:ConesFunctorial}
This section explains the functoriality of the foliated cone construction. As mentioned in the introduction, only the examples at the end of this section are necessary for understanding the rest of this paper.

Several equivalent definitions of the notion of smooth maps between leaf spaces are given in {\cite[Section I]{HilSka}}.
In this paper, we prefer to use the following hands-on definition:
\begin{mdef}[({\cite[Section I]{HilSka}})]\label{def:smoothmap}
Let $(M_1,\cF_1),(M_2,\cF_2)$ be two foliations and denote their graphs by $G_1,G_2$. A \emph{smooth map $f:\,M_1/\cF_1\to M_2/\cF_2$ between the leaf spaces}
is given by
\begin{itemize}
\item an open cover $\{\Omega_\alpha\}_{\alpha\in I}$ of $M_1$ and
\item a collection of smooth maps $f_{\alpha\beta}:(G_1)^{\Omega_\alpha}_{\Omega_\beta}\to G_2$
\end{itemize}
such that
\[\forall\gamma\in (G_1)^{\Omega_\alpha}_{\Omega_\beta}:\,f_{\beta\alpha}(\gamma^{-1})=f_{\alpha\beta}(\gamma)^{-1}\]
and for all $\gamma_1\in (G_1)^{\Omega_{\alpha_1}}_{\Omega_{\alpha_2}},\gamma_2\in (G_1)^{\Omega_{\alpha_2}}_{\Omega_{\alpha_3}}$ with $s(\gamma_1)=r(\gamma_2)$ we have 
\[s(f_{\alpha_1\alpha_2}(\gamma_1))=r(f_{\alpha_2\alpha_3}(\gamma_2))\] and
\[f_{\alpha_1\alpha_2}(\gamma_1)f_{\alpha_2\alpha_3}(\gamma_2)=f_{\alpha_1\alpha_3}(\gamma_1\gamma_2).\]
\end{mdef}
Note that the open cover $\{\Omega_\alpha\}_{\alpha\in I}$ is a somewhat unwelcome datum in the definition, as the notion of a smooth map between leaf spaces should not depend on its choice. 
%Indeed, alternate definitions (cf. \cite[Section I]{HilSka}) do not inclose this datum.
For this reason, we shall identify smooth maps under the following equivalence relation:
\begin{mdef}\label{def:equivalenceofsmoothmaps}
Two smooth maps $f,f':\,M_1/\cF_1\to M_2/\cF_2$ between leaf spaces which are given by the sets of data 
\[\left(\{\Omega_\alpha\}_{\alpha\in I},\{f_{\alpha\beta}\}_{(\alpha,\beta)\in I\times I}\right),\quad\left(\{\Omega'_\alpha\}_{\alpha\in I'},\{f'_{\alpha\beta}\}_{(\alpha,\beta)\in I'\times I'}\right),\]
respectively,  
are defined to be equivalent iff they are obtained by restriction of index sets from another smooth map 
$\tilde f:\,M_1/\cF_1\to M_2/\cF_2$ 
which is of the form
\[\left(\{\Omega_\alpha\}_{\alpha\in I}\dot\cup\{\Omega'_\alpha\}_{\alpha\in I'},\{\tilde f_{\alpha\beta}\}_{(\alpha,\beta)\in(I\dot\cup I')\times(I\dot\cup I')} \right).\]
\end{mdef}

Given a smooth map $f:\,M_1/\cF_1\to M_2/\cF_2$ between leaf spaces as in Definition \ref{def:smoothmap}, the unit map $u:M_1\to G_1$ restricts to  $u|_\alpha:\,\Omega_\alpha\to (G_1)_{\Omega_\alpha}^{\Omega_\alpha}$. The second property shows that the image of the composition $f_{\alpha\alpha}\circ u|_\alpha$ lies in the unit space $M_2$ of $G_2$. We thus obtain a family of smooth maps
\[f_\alpha:=f_{\alpha\alpha}\circ u|_\alpha:\,\Omega_{\alpha}\to M_2.\]
The second property in the definition also implies 
\begin{equation}\label{eq:smoothleafmaps}
\forall\gamma\in (G_1)^{\Omega_\alpha}_{\Omega_\beta}:\quad r(f_{\alpha\beta}(\gamma))=f_\alpha(r(\gamma)),\quad s(f_{\alpha\beta}(\gamma))=f_\beta(s(\gamma)).
\end{equation}
In particular, any representative of $f_{\alpha\beta}(u(x))\in G_2$ for $x\in \Omega_\alpha\cap\Omega_\beta$ is a leafwise path in $M_2$ between $f_\alpha(x)$ and $f_\beta(x)$.

Furthermore, \eqref{eq:smoothleafmaps} implies that if $c$ is a path in $\Omega_{\alpha}$ which is contained in a single leaf then $f_{\alpha}\circ c$ is also contained in a single leaf.  Thus, $df_\alpha$ maps $T_x\cF_1$ to $T_{f_\alpha(x)}\cF_2$.

Finally, composition of two smooth maps $f:\,M_1/\cF_1\to M_2/\cF_2$, $g:\,M_2/\cF_2\to M_3/\cF_3$ works, of course, as follows. Let 
\[\left(\{\Omega_\alpha\}_{\alpha\in I},\{f_{\alpha\beta}\}_{(\alpha,\beta)\in I\times I}\right),\quad\left(\{\Theta_\xi\}_{\xi\in J},\{g_{\xi\zeta}\}_{(\xi,\zeta)\in J\times J}\right)\]
be the sets of data defining $f,g$, respectively. Then $g\circ f$ is defined by the open cover of $G_1$ consisting of the sets $\Phi_{(\alpha,\xi)}:=f_\alpha^{-1}(\Theta_\xi)$ with ${(\alpha,\xi)\in I\times J}$ and the compositions
$g_{\xi\zeta}\circ f_{\alpha\beta}$ which are defined on the sets $(G_1)^{\Phi_{(\alpha,\xi)}}_{\Phi_{(\beta,\zeta)}}$.

The main result of this section is the following functoriality theorem.
\begin{thm}\label{thm:conefunctoriality}
The foliated cone construction is a functor from the category of foliations and smooth maps of leaf spaces between them to the category of metrizable coarse spaces and coarse equivalence classes of coarse Borel maps between them.
\end{thm}

\begin{proof}
Assume we are given two compact foliations $(M_1,\cF_1),(M_2,\cF_2)$ 
and a smooth map $f:M_1/\cF_1\to M_2/\cF_2$ between the leaf spaces consisting of the data specified in Definition \ref{def:smoothmap}. 
The first step of the proof is to define the induced map $f_*:\cO(M_1,\cF_1)\to\cO(M_2,\cF_2)$ between the foliated cones. 

Before we start making some choices, let us fix  Riemannian metrics $g_1,g_2$ on $M_1,M_2$ which we shall use for the construction of the foliated cones $\mathcal{O}(M_1,\cF_1)$, $\mathcal{O}(M_2,\cF_2)$, respectively.

Because of compactness of $M_1$, we can find a finite open cover $(\Omega'_i)_{i=1,\dots,m}$ such that the closure of each $\Omega'_i$ is compact and contained in some $\Omega_{\alpha(i)}$. Denote the restriction of $f_{\alpha(i)}$ to $\overline{\Omega'_i}$ by $f_i$ and the restriction of $f_{\alpha(i)\alpha(j)}$ to $(G_1)^{\overline{\Omega'_i}}_{\overline{\Omega'_j}}$ by $f_{ij}$.
Furthermore, choose a Borel map $i:M_1\to \{1,\dots,m\},x\mapsto i_x$ such that $\forall x\in M_1: x\in \Omega'_{i_x}$.
\begin{prop}\label{prop:inducedmap}
With notations as above, the map 
\[f_*:\,\mathcal{O}(M_1,\cF_1)\to\mathcal{O}(M_2,\cF_2),\,(t,x)\mapsto (t,f_{i_x}(x))\]
is a coarse Borel map. The coarse equivalence class of this map is independent of the choices made. Furthermore, the coarse equivalence class only depends on the equivalence class of the smooth map $f$.
\end{prop}
\begin{proof}
The following observation is central to the proof.
\begin{lem}\label{lemma:jumplength}
There is a constant $L$ such that for all $1\leq i,j\leq m$ and $x\in\Omega'_i\cap\Omega'_j$ the points $f_i(x),f_j(x)$ are joined by a leafwise path of length at most $L$. In particular, the points 
\[(t,f_i(x)),(t,f_j(x))\in\mathcal{O}(M_2,\cF_2)\]
have distance at most $L$ for all $t\geq 0$.
\end{lem}
\begin{proof}
The length function on the holonomy groupoid $G_2$ is bounded on the compact subset
\[\bigcup_{1\leq i,j\leq m}f_{ij}(u(\overline{\Omega'_i}\cap\overline{\Omega'_j}))\]
by Corollary \ref{length:boundedoncompact}. As we observed above, the points $f_i(x),f_j(x)$ are joined by any representative of some element in this compact set.
\end{proof}
Because of this Lemma, the choice of the Borel map $i:M_1\to \{1,\dots,m\}$ affects the map $f_*$ only by a distance of at most $L$ and therefore the coarse equivalence class remains unchanged. 
Independence of the finite and compact refinement $(\Omega'_i)_{i=1,\dots,m'}$ is clear, as this family may be enlarged by a finite number of open sets while leaving the Borel map $i$ and therefore also $f_*$ unchanged.
For exactly the same reason, equivalent smooth maps between the leaf spaces also induce coarsely equivalent maps between the foliated cones.

It remains to show that  $f_*$ is a coarse map. Properness is clear, but the expansion condition has to be verified.
We start with a little calculation.
\begin{lem}\label{lemma:pathpiecelength}
There is a constant $K$ such that whenever $i\in\{1,\dots,m\}$ and $c:[0,1]\to [0,\infty)\times \Omega'_i$ is  a smooth path then the length estimate
\[L((id\times f_i)\circ c)\leq K\cdot L(c)\]
holds. Here, the length of the curve $c$ is measured by equipping $[0,\infty)\times \Omega'_i$ with the restricted metric as a subset of $\cO(M_1,\cF_1)$, and the length of $(id\times f_i)\circ c$ is, of course, measured in $\cO(M_2,\cF_2)$.
\end{lem}
\begin{proof}
Let \[K':=\max_{i=1,\dots,m}\sup_{x\in\overline{\Omega'_i}}\|df_i(x)\|\,.\]
Consider a tangent vector $\xi\in T_{(t,x)}\mathcal{O}(M_1,\cF_1)$ at a point $(t,x)\in[0,\infty)\times\Omega'_i$. We write
$\xi=\lambda\frac{\partial}{\partial t}+\xi_{\|}+\xi_\perp$
according to the orthogonal decomposition 
\[T_{(t,x)}\mathcal{O}(M_1,\cF_1)=\R\frac{\partial}{\partial t}\oplus T_x\cF_1\oplus N_x\cF_1.\]
Its norm is given by 
\[\|\xi\|^2=\lambda^2+\|\xi\|^2+t^2\|\xi_\perp\|^2=\lambda^2+\|\xi_\|\|^2+(1+t^2)\|\xi_\perp\|^2.\]
Now recall that $df_i(\xi_\|)$ is tangent to $\cF_2$. Thus, we can calculate
\begin{align*}
\|df_i(\xi)\|^2&=\|\lambda\frac{\partial}{\partial t}+df_i(\xi_{\|})+df_i(\xi_\perp)\|^2
\\&=\lambda^2+\|df_i(\xi_{\|})+df_i(\xi_\perp)_\|\|^2+(1+t^2)\|df_i(\xi_\perp)_\perp\|^2
\\&\leq\lambda^2+(\|df_i(\xi_{\|})\|+\|df_i(\xi_\perp)\|)^2+(1+t^2)\|df_i(\xi_\perp)\|^2
\\&\leq \lambda^2+K'^2(\|\xi_{\|}\|+\|\xi_\perp\|)^2+K'^2(1+t^2)\|\xi_\perp\|^2
\\&\leq \lambda^2+2K'^2\|\xi_{\|}\|^2+K'^2(3+t^2)\|\xi_\perp\|^2
\end{align*}
and now we see that $\|df_i(\xi)\|\leq K\|\xi\|$ for
\[K=\sqrt{\max(1,3K'^2)}.\]
The claim follows.
\end{proof}
Let $R>0$. We have to find $S>0$ such that whenever $z,z'\in \mathcal{O}(M_1,\cF_1)$ are points of distance less than $R$ then the distance between $f_*(z),f_*(z')$ is less than $S$. We can estimate the distance between $f_*(z),f_*(z')$ by a sequence of paths as in Lemma \ref{lemma:pathpiecelength} and jumps between points $(t,f_i(x)),(t,f_j(x))$ of length $\leq L$ as in Lemma \ref{lemma:jumplength}.
The only thing we need more is an upper bound for the number of jumps needed.
\begin{lem}\label{lemma:jumpnumber}
Let $R>0$. There is $k\in \N$ such that whenever $c:[0,1]\to M_1$ is a path of length $\leq R$, then there are $0= s_0\leq\dots\leq s_k= 1$ such that for each $l=0,\dots,k-1$ the image of  $c|_{[s_l,s_{l+1}]}$ is contained in some $\Omega'_{i(l)}$.
\end{lem}
\begin{proof}
Let $\varepsilon>0$ be such that every $\varepsilon$-ball in $M_1$ is contained in some $\Omega'_i$. Then the claim follows easily for some fixed $k>R/2\varepsilon$.
\end{proof}
Now if $z,z'\in \mathcal{O}(M_1,\cF_1)$ are less than distance $R$ apart, then they are joined by a path 
\[(t,c):\,[0,1]\to\mathcal{O}(M_1,\cF_1)\]
of length less than $R$. It follows that $c:[0,1]\to M_1$ has length less than $R$ (measured in $g_1$), too, and we can apply  Lemma \ref{lemma:jumpnumber}. 
With $s_l,i(l)$ as in the lemma, we can now go from 
\[f_*(z)=(t(s_0),f_{i_{c(s_0)}}(c(s_0)))\quad\text{ to }\quad f_*(z')=(t(s_k),f_{i_{c(s_k)}}(c(s_k)))\]
by first jumping to the point $(t(s_0),f_{i(0)}(c(s_0)))$, then following the path 
\[(t|_{[s_0,s_1]},\,f_{i(0)}\circ c|_{[s_0,s_1]})\]
to the point $(t(s_1),f_{i(0)}(c(s_1)))$, then jumping again to $(t(s_1),f_{i(1)}(c(s_1)))$ and so on. We reach the endpoint after $k+1$ jumps of length at most $L$ and $k$ smooth paths in between, whose total length is at most $K\cdot R$.

Thus, $S:=(k+1)L+KR$ has the desired properties. This finishes the proof of Proposition \ref{prop:inducedmap}.
\end{proof}

An easy corollary, which one could also prove directly, is the following.
\begin{cor}
The coarse structure of the foliated cone $\cO(M,\cF)$ is independent of the chosen metric on $M$.
\end{cor}
\begin{proof}
Applying the Proposition to the smooth map of leaf spaces $M/\cF\to M/\cF$ given by the one element open cover $\{M\}$ of $M$ and the map $\id:G=G_M^M\to G$ shows that the identity between the cones on the left and right hand side, which are allowed to be constructed with different metrics on $M$, is always a coarse map.
\end{proof}

With induced maps being defined as in Proposition \ref{prop:inducedmap},
functoriality is in fact obvious, provided one makes the canonical choices in the construction of $(g\circ f)_*$.
This finishes the proof of  Theorem \ref{thm:conefunctoriality}.
\end{proof}

We conclude this section with easy but fundamental examples of such induced coarse maps. 

\begin{example}\label{ex:smoothmaps}
There is an obvious functor from the category of smooth compact manifolds to the category of compact foliations and smooth maps of leaf spaces: A manifold $M$ is mapped to the trivial $0$-dimensional foliation $(M,\cF_0^M)$ on $M$.
A smooth map $f:M\to N$ induces the smooth map of leaf spaces $f:M/\cF^M_0\to N/\cF^N_0$ given by the following data: The graphs of $(M,\cF^M_0)$, $(N,\cF^N_0)$ are $M,N$ themselves and $f:M/\cF^M_0\to N/\cF^N_0$ is given by the one element cover ${M}$ and the smooth function $f:M\to N$.
The induced map $\cO M\to \cO N$ is the obvious one: $\id_{[0,\infty)}\times f$.
\end{example}

\begin{example}\label{ex:projtoleafspace}
Given a compact foliation $(M,\cF)$, there is a canonical smooth map of leaf spaces
\[p:M\to M/\cF,\]
the left hand side being $M/\cF_0$ for the trivial $0$-dimensional foliation $\cF_0$ on $M$.
The graph of $(M,\cF_0)$ is simply $M$ while we denote the graph of $(M,\cF)$ by $G$.
The map $p$ consists of the one element open cover $\{M\}$ of $M$ and the unit map  $M=M_M^M\to G$.

The induced coarse map $\mathcal{O}M\to \cO(M,\cF)$ is simply the identity on the underlying topological space $[0,\infty)\times M$. It is obviously $1$-Lipschitz. 
\end{example}

\begin{example}\label{example:submersionsmoothmaps}
Assume that the foliation $(M,\cF)$ comes from a 
submersion $p:M\to B$. 
This submersion factors through a smooth map of leaf spaces $\tilde p:M/\cF\to B$ which consists of the one element covering $\{M\}$ of $M$ and the map
\[p\circ s=p\circ r:\,G=G_M^M\to B\]
where $G$
 is the holonomy groupoid of $(M,\cF)$ and $B$ is the holonomy groupoid of the trivial $0$-dimensional foliation on $B$.
The induced coarse map we obtain is simply
\[\tilde p_*=\id_{[0,\infty)}\times p:\,\cO(M,\cF)\to\mathcal{O}B.\]

Now assume that $p$ is surjective and all  fibers of $p$ are connected. In this case the fibers are uniformly bounded when measured by the length of smooth leafwise paths. As leafwise measured distances do not blow up in the foliated cone, we see that $\tilde p_*$ is a coarse equivalence.

This coarse equivalence is not surprising, because $\tilde p$ already is an isomorphism in the category of foliations and equivalence classes of smooth maps between them.
The inverse $f:B\to M/\cF$ of $\tilde p$ is defined as follows. Let $\{\Omega_\alpha\}$ be an open cover of $B$ such that for each $\alpha$ there is a smooth section $s_\alpha:\Omega_\alpha\to p^{-1}(\Omega_\alpha)$ of $p$ over $\Omega_\alpha$. 
Noting that the holonomy groupoid of $(M,\cF)$ in this special case is 
\[G=M\times_B M=\{(x,y)\in M\times M\,|\,p(x)=p(y)\},\] we can now define $f$ by the cover $\{\Omega_\alpha\}$ and the maps 
\[f_{\alpha\beta}=(s_\alpha,s_\beta):\Omega_\alpha\cap\Omega_\beta\to M\times_B M.\]
One readily checks that $f\circ \tilde p$ and $\tilde p\circ f$ are indeed equivalent to the identities.

\end{example}

\section{The new $K$-theory model}\label{sec:newKtheory}
Finally, we have all prerequisites for introducing the new $K$-theory model.
\begin{mdef}\label{def:FJModel}
The ``Farrell-Jones'' model for the $K$-theory of the leaf space of a foliation $\fol$ is
\[K^{-*}_{FJ}(M/\mathcal{F}):=K_*(\cfrak(\FC)).\]
We also define the $K$-theory with coefficients in a $C^*$-algebra $D$ by
\[K^{-*}_{FJ}(M/\mathcal{F},D):=K_*(\cfrak(\FC,D)).\]
\end{mdef}
It will be more convenient to perform proofs using the even more general groups $K_*(\kfrak(\FC,D))$. However, we will not give them any special name.

\begin{rem}\label{rem::compactification:corona}
A practical consequence arises from the fact that in  the definition of the foliated cone $\FC$ we did not crush $\{0\}\times M$ to a point:
The $C^*$-algebra $C_0(\FC)\cong C_0([0,\infty)\times M)$ is contractible and therefore the short exact sequence
\[0\to C_0([0,\infty)\times M)\otimes D\to\kbar(\FC,D)\to\kfrak(\FC,D)\to 0\]
implies that the quotient maps induce canonical isomorphisms
\[K_*(\kbar(\FC,D))\cong K_*(\kfrak(\FC,D)).\]
Indeed, working with the compactification instead of the corona makes calculating the examples in the following section much easier. 
\end{rem}

The rest of this section is devoted to showing basic properties of the ``Farrell-Jones'' $K$-theory which justify why it is a good $K$-theory model.

One might expect that $K$-theory models of ``spaces'' have some ring structure and contravariant functoriality, and indeed:
\begin{thm}\label{thm:Ktheorymodelfunctor}
The groups $K^*_{FJ}(M/\mathcal{F})$ constitute a contravariant functor from the category of foliations and smooth maps between leaf spaces into the category of $\Z_2$-graded, graded commutative rings. 
The obvious analogous statements hold for the more general
 groups 
$K^*_{FJ}(M/\mathcal{F},D)$ and $K_*(\kfrak(\FC,D))$.
\end{thm}
\begin{proof}
Functoriality is the result of combining Theorem \ref{thm:conefunctoriality} with Proposition \ref{prop:coronasfunctorial}.
The multiplicative structures are provided by Definition \ref{def:coronaproduct}.
\end{proof}

Furthermore, we mention the following two examples, which serve as a basic test for the Farrell-Jones $K$-theory model.
\begin{example}\label{ex:trivialfoliation}
Let $M$ be a compact manifold. 
From \cite[Proposition 6.2]{WulffCoassemblyRinghomo} we know that the inclusion $C(M)\otimes\K\subset\cfrak(\mathcal{O}M)$ induces a ring isomorphism
\[K^*(M)\xrightarrow{\cong}K^*_{FJ}(M),\]
the right hand side being $K^*_{FJ}(M/\cF)$ for the trivial $0$-dimensional foliation on $M$. 
Using Example \ref{ex:smoothmaps} one readily checks that this isomorphism is natural under smooth maps of manifolds.
Thus, our $K$-theory of leaf spaces extends $K$-theory of ordinary manifolds.

Furthermore, note that the identification of \ref{rem::compactification:corona} allows an easy description of a (right-sided and therefore also two-sided) inverse to this isomorphism: It is simply induced by the restriction
\[\cbar(\cO M)\to C(M)\otimes\K,\quad f\mapsto f|_{\{0\}\times M}\,.\]
\end{example}

\begin{example}\label{ex:submersion}
Assume that the foliation $\fol$ comes from a surjective submersion $p:M\to B$ with all fibers connected. 
We saw in Example \ref{example:submersionsmoothmaps} that there is an induced isomorphism of leaf spaces $\tilde p:M/\mathcal{F}\cong B$ inducing the coarse equivalence $\tilde p_*=\id_{[0,\infty)}\times p:\FC\to\cO B$ between the corresponding foliated cones.

Combining this with the previous example, there is a canonical ring isomorphism
\[K^*(B)\cong K^*_{FJ}(B)\xrightarrow[\cong]{\tilde p^*}K^*_{FJ}(M/\mathcal{F})\]
induced by the inclusion 
\[C(B)\otimes\K\hookrightarrow \cfrak(\FC),\quad  g\mapsto \overline{ (t,x)\mapsto p^*g(x)}.\]
Here and in the following, the overline denotes the equivalence class.

Analogously there are multiplicative homomorphisms
\[K_{-*}(C(B)\otimes D)\cong K_{FJ}^*(B,D)\cong K_{FJ}^*(M/\mathcal{F},D) \]
for any coefficient $C^*$-algebra $D$.

\end{example}

\section{Nontrivial examples}\label{sec:examples}
In this section, we give some more examples of nontrivial elements in the ring $K^*_{FJ}(M/\mathcal{F})$.

In general, one way of doing so is by constructing a continuous map $\phi:\FC\to X$ of vanishing variation into some compact metric space $X$ and using it to pull back elements of $K^*(X)$. 
More precisely, there is an induced $*$-homomorphism $\phi^*:C(X)\otimes\K\to\cfrak(\FC)$ and subsequently a homormorphism $\phi^*:K^*(X)\to K^*_{FJ}(M/\cF)$. Note that $\phi^*$ is in fact a ring homomorphism, because multiplicativity follows from the commutative diagram
\[\xymatrix{
\FC\ar[d]_{\Delta}\ar[rr]^{\phi}&&\ar[d]^{\Delta}X
\\\FC\times\FC\ar[rr]_{\phi\times\phi}&&X\times X.
}\]

We would also like to have a method of distinguishing elements of $K^*_{FJ}(M/\mathcal{F})$. In some situations, an effective way to do this is to use the homomorphism
\[p^*:K^*_{FJ}(M/\mathcal{F})\to K^*(M)\]
induced by the canonical smooth map $p:M\to M/\mathcal{F}$ of Example \ref{ex:projtoleafspace}.
After identifying $K^*_{FJ}(M/\mathcal{F})$ with $K_{-*}(\cbar(\FC))$ and applying the isomorphism from Example \ref{ex:trivialfoliation}, one easily sees that $p^*$ is induced by the restriction $*$-homomorphism
\begin{equation}\label{eq:pstar}
\cbar(\FC)\to C(M)\otimes \K,\quad f\mapsto f|_{\{0\}\times M}\,.
\end{equation}
Note that we would not have this simple formula if we had stuck to the stable Higson corona instead of the stable Higson compactification.

\begin{example}
Consider the one dimensional foliation of the 2-torus sketched in figure \ref{fig:TorusFoliation}.
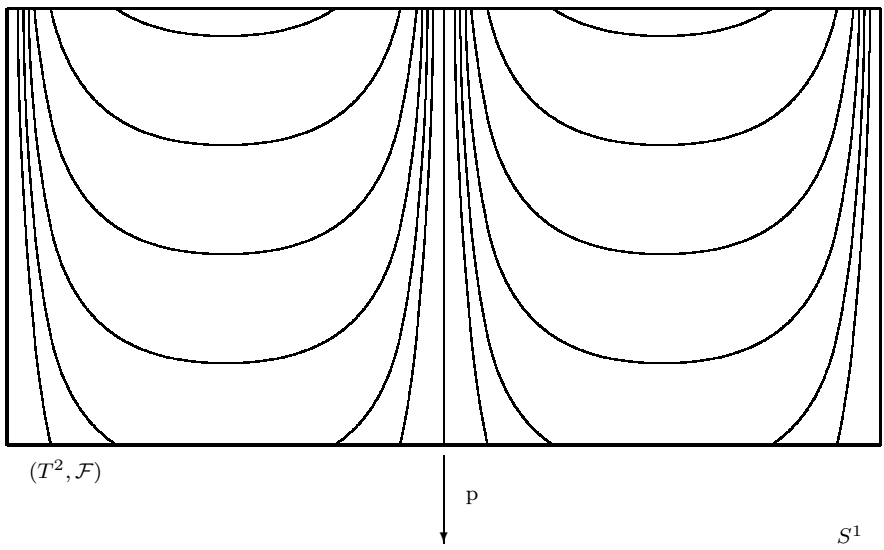
\begin{figure}[htbp]
\begin{center}
\setlength{\unitlength}{0.1\textwidth}
\begin{picture}(9,5)(-0.5,-1)
\thicklines
\put(0,0){\line(1,0){8}}
\put(0,4){\line(1,0){8}}
\put(0,0){\line(0,1){4}}
\put(8,0){\line(0,1){4}}

\put(0,-1){\line(1,0){8}}

\thinlines

\qbezier(2,0.75)(1.375,0.75)(1,1)%
\qbezier(2,1.75)(1.375,1.75)(1,2)
\qbezier(2,2.75)(1.375,2.75)(1,3)
\qbezier(2,3.75)(1.375,3.75)(1,4)

\qbezier(2,0.75)(2.625,0.75)(3,1)%
\qbezier(2,1.75)(2.625,1.75)(3,2)
\qbezier(2,2.75)(2.625,2.75)(3,3)
\qbezier(2,3.75)(2.625,3.75)(3,4)

\qbezier(1,1)(0.55,1.3)(0.4,2)%
\qbezier(1,2)(0.55,2.3)(0.4,3)
\qbezier(1,3)(0.55,3.3)(0.4,4)
\qbezier(1,0)(0.55,0.3)(0.4,1)

\qbezier(3,1)(3.45,1.3)(3.6,2)%
\qbezier(3,2)(3.45,2.3)(3.6,3)
\qbezier(3,3)(3.45,3.3)(3.6,4)
\qbezier(3,0)(3.45,0.3)(3.6,1)

\qbezier(0.4,2)(0.3,2.47)(0.25,3)%
\qbezier(0.4,3)(0.3,3.47)(0.25,4)
\qbezier(0.4,0)(0.3,0.47)(0.25,1)
\qbezier(0.4,1)(0.3,1.47)(0.25,2)

\qbezier(3.6,2)(3.7,2.47)(3.75,3)%
\qbezier(3.6,3)(3.7,3.47)(3.75,4)
\qbezier(3.6,0)(3.7,0.47)(3.75,1)
\qbezier(3.6,1)(3.7,1.47)(3.75,2)

\qbezier(0.25,3)(0.2,3.53)(0.2,4)
\qbezier(0.25,2)(0.15,3.06)(0.15,4)
\qbezier(0.25,1)(0.1,2.59)(0.1,4)

\qbezier(3.75,3)(3.8,3.53)(3.8,4)
\qbezier(3.75,2)(3.85,3.06)(3.85,4)
\qbezier(3.75,1)(3.9,2.59)(3.9,4)

\put(4,0){\line(0,1){4}}

\qbezier(6,0.75)(5.375,0.75)(5,1)%
\qbezier(6,1.75)(5.375,1.75)(5,2)
\qbezier(6,2.75)(5.375,2.75)(5,3)
\qbezier(6,3.75)(5.375,3.75)(5,4)

\qbezier(6,0.75)(6.625,0.75)(7,1)%
\qbezier(6,1.75)(6.625,1.75)(7,2)
\qbezier(6,2.75)(6.625,2.75)(7,3)
\qbezier(6,3.75)(6.625,3.75)(7,4)

\qbezier(5,1)(4.55,1.3)(4.4,2)%
\qbezier(5,2)(4.55,2.3)(4.4,3)
\qbezier(5,3)(4.55,3.3)(4.4,4)
\qbezier(5,0)(4.55,0.3)(4.4,1)

\qbezier(7,1)(7.45,1.3)(7.6,2)%
\qbezier(7,2)(7.45,2.3)(7.6,3)
\qbezier(7,3)(7.45,3.3)(7.6,4)
\qbezier(7,0)(7.45,0.3)(7.6,1)

\qbezier(4.4,2)(4.3,2.47)(4.25,3)%
\qbezier(4.4,3)(4.3,3.47)(4.25,4)
\qbezier(4.4,0)(4.3,0.47)(4.25,1)
\qbezier(4.4,1)(4.3,1.47)(4.25,2)

\qbezier(7.6,2)(7.7,2.47)(7.75,3)%
\qbezier(7.6,3)(7.7,3.47)(7.75,4)
\qbezier(7.6,0)(7.7,0.47)(7.75,1)
\qbezier(7.6,1)(7.7,1.47)(7.75,2)

\qbezier(4.25,3)(4.2,3.53)(4.2,4)
\qbezier(4.25,2)(4.15,3.06)(4.15,4)
\qbezier(4.25,1)(4.1,2.59)(4.1,4)

\qbezier(7.75,3)(7.8,3.53)(7.8,4)
\qbezier(7.75,2)(7.85,3.06)(7.85,4)
\qbezier(7.75,1)(7.9,2.59)(7.9,4)

\put(4,-0.1){\vector(0,-1){0.8}}
\put(4.2,-0.5){\mbox{p}}
\put(7.6,-0.9){$S^1$}
\put(0.2,-0.3){$(T^2,\mathcal{F})$}

\end{picture}
\caption{A one dimensional foliation on the two torus.}
\label{fig:TorusFoliation}
\end{center}
\end{figure}
The slices $T^2\times\{t\}\subset\mathcal{O}(T^2,\mathcal{F})$ of the foliated cone become larger and larger in the horizontal direction as $t\to \infty$. Given a unitary over $C(S^1)$, we can pull it back to $T^2$ via the projection $p$ onto the horizontal $S^1$. Subsequently, the variation of the unitary may be pushed into small neighborhoods of the compact leaves, where the metric blows up horizontally. By doing this, we obtain a unitary in $\widetilde{\cbar(\FC)}$ and thus an element of $K_1(\cfrak(\FC))$ which is kind of a pullback of an element of $K^1(S^1)$ under the projection $p$. 

The precise calculations involved are quite elaborate. We will perform them for a more general setup in Example \ref{example:Sn}.
\end{example}

\begin{example}\label{example:Sn}
Let $\mathcal{F}$ be a one dimensional foliation on some compact manifold $M$ and $L$ a leaf of this foliation which is diffeomorphic to $S^1=\R/\Z$. Assume that the normal bundle of $L$ is trivial, such that there is a tubular neighborhood of $L$ diffeomorphic to $D^n\times S^1$ in which $L$ corresponds to $\{0\}\times S^1$. Assume further that within this neighborhood the foliation is given by the trajectories of a unit vector field of the form
\[v(x,s)=\left(\lambda(x)x,\sqrt{1-\lambda(x)^2\|x\|^2}\right)\in \R^n\times\R\cong T_{(x,s)}(D^n\times S^1)\]
 for some continuous function $\lambda:D^n\to\R$. In particular, the vector field is $S^1$-invariant in this neighborhood.

The objectives of this example are to construct a ring homomorphism $K^*(S^n)\cong\Z[X]/(X^2)\to K^*_{FJ}(M/\cF)$ and
to show that it is injective for $M=S^n\times S^1$. Thus, it is an example with nontrivial ring structure which is quite different from example \ref{ex:submersion}.

We are free to choose any Riemannian metric $g$ on $M$ to construct the foliated cone. Therefore, we may assume without loss of generality that it is the canonical one on the tubular neighborhood $D^n\times S^1$.

The first step is to construct a continuous map of vanishing variation $\Phi:\FC\to S^n$ as follows.
Consider the map 
\[\phi:D^n\times S^1\times [0,\infty)\to \R^n,\quad(x,s,t)\mapsto\begin{cases}\frac{\log(t\|x\|+1)}{\log(t+1)}\frac{x}{\|x\|}&t>0\\x&t=0.\end{cases}\]
It is smooth and maps $S^{n-1}\times S^1\times[0,\infty)$ to $S^{n-1}$. Furthermore, let $\exp:\R^n\to S^n$ be the exponential map at the north pole $e$ of the sphere. It maps $\pi\cdot S^{n-1}$ to the south pole $-e$ and is Lipschitz continuous with constant 1.

\begin{lem}
The continuous map
\begin{align*}
\Phi:\FC&\to S^n
\\x&\mapsto\begin{cases}\exp(\pi\cdot \phi(x))&x\in D^n\times S^1\times[0,\infty)\\-e&\text{else}\end{cases}
\end{align*}
has vanishing variation.
\end{lem}
\begin{proof}
Note that it is enough to show that $\phi$ has vanishing variation with respect to the restricted metric on $D^n\times S^1\times [0,\infty)$. To this end, let $w=(\xi,\mu)\in \R^n\times\R\cong T_{(x,s)}(D^n\times S^1)$. Denote by $w_{\|}=\langle w,v(x,s)\rangle\cdot v(x,s)$ its  component tangential to the leaves. Furthermore, we decompose the $\R^n$-component $\xi=\xi_\perp+\xi_\|$ of $w$ into a component $\xi_\|:=\frac{\langle\xi,x\rangle}{\|x\|^2}x$ parallel to $x$ and a component $\xi_\perp$ perpendicular to it. In the norm corresponding to the Riemannian metric $g_t:=g+t^2g_N$, we have
\begin{align*}
\|w\|^2_t&=\|w\|^2+t^2\|w-w_{\|}\|^2
\\&=(1+t^2)\|w\|^2-2t^2\langle w,w_{\|}\rangle+t^2\|w_{\|}\|^2
\\&=(1+t^2)\|w\|^2-t^2\langle w,v(x,s)\rangle^2
\\&=(1+t^2)(\|\xi\|^2+\mu^2)-t^2\left(\lambda(x)\langle\xi,x\rangle+\mu\sqrt{1-\lambda(x)^2\|x\|^2}\right)^2.
\end{align*}
By minimizing this quadratic expression in $\mu$ and defining  $L:=\|\lambda\|_\infty$,
 we obtain the inequality
\[\|w\|^2_t\geq(1+t^2)\|\xi_\perp\|^2+\frac{1+t^2}{1+t^2L^2\|x\|^2}\|\xi_{\|}\|^2.\]
Thus, the norm of the tangential vector $w+\eta\frac{\partial}{\partial t}=(\xi,\mu,\eta)\in\R^n\times\R\times\R\cong T_{(x,s,t)}\FC$ is bounded from below by
\[\left\|w+\eta\frac{\partial}{\partial t}\right\|^2\geq(1+t^2)\|\xi_\perp\|^2+\frac{1+t^2}{1+t^2L^2\|x\|^2}\|\xi_{\|}\|^2+\eta^2.\]
On the other hand, we define $f(r,t):=\frac{\log(rt+1)}{\log(t+1)}$ and calculate
\[
D\phi\left(w+\eta\frac{\partial}{\partial t}\right)=f(\|x\|,t)\frac{\xi_\perp}{\|x\|}+\frac{\partial f}{\partial r}(\|x\|,t)\xi_{\|}+\eta\frac{\partial f}{\partial t}(\|x\|,t)\frac{x}{\|x\|}.
\]
Thus,
\begin{align*}
&\left\|D\phi\left(w+\eta\frac{\partial}{\partial t}\right)\right\|^2=\frac{f(\|x\|,t)^2}{\|x\|^2}\|\xi_\perp\|^2+\left(\frac{\partial f}{\partial r}(\|x\|,t)\|\xi_{\|}\|\pm\frac{\partial f}{\partial t}(\|x\|,t)\eta\right)^2
\\&\qquad\leq\left( \frac{f(\|x\|,t)^2}{(1+t^2)\|x\|^2}+2\left(\frac{\partial f}{\partial r}(\|x\|,t)\right)^2\frac{1+t^2L^2\|x\|^2}{1+t^2}
+2\left(\frac{\partial f}{\partial t}(\|x\|,t)\right)^2\right)\left\|w+\eta\frac{\partial}{\partial t}\right\|^2.
\end{align*}
Vanishing variation of $\phi$ therefore follows from the fact that the three expressions
\begin{align*}
 \frac{f(r,t)^2}{(1+t^2)r^2}&=\frac{\log(rt+1)^2}{\log(t+1)^2(1+t^2)r^2}
 \\\left(\frac{\partial f}{\partial r}(r,t)\right)^2\frac{1+t^2L^2r^2}{1+t^2}&=\frac1{\log(t+1)}\cdot\frac{t^2}{1+t^2}\cdot\frac{1+t^2L^2r^2}{(tr+1)^2}
 \\\frac{\partial f}{\partial t}(r,t)&=\frac{r}{(rt+1)\log(t+1)}-\frac{\log(rt+1)}{(t+1)\log(t+1)^2}
\end{align*}
converge to $0$ uniformly in $r\in(0,1]$ for $t\to \infty$, as is readily verified.
\end{proof}
According to our remarks at the beginning of this section we obtain:
\begin{cor}
The map $\Phi$ induces  a ring homomorphism \[\Phi^*:K^*(S^n)\to K^*_{FJ}(M/\mathcal{F}).\]
\end{cor}
Furthermore, the composition $p^*\circ\Phi^*:K^*(S^n)\to K^*(M)$ is induced by the continuous map 
\[\psi:M\to S^n,\quad x\mapsto \begin{cases} \exp(\pi\cdot y)&\text{if }x=(y,s)\in D^n\times S^1\\-e&\text{else}.\end{cases}\]
If we specialize to the case $M=S^n\times S^1$ where $D^n\times S^1\subset M$ is assumed to come from an inclusion $D^n\subset S^n$, then the map $\psi$ is homotopic to the canonical projection $S^n\times S^1\to S^n$. Thus, $p^*\circ\Phi^*=\psi^*:K^*(S^n)\to K^*(S^n\times S^1)$ is injective. 
 In particular, the ring homomorphism 
\[\Phi^*:K^*(S^n)\cong\Z[X]/(X^2)\hookrightarrow K^*_{FJ}(M/\mathcal{F})\]
 is injective. We have thus  detected some nontrivial ring structure inside of $K^*_{FJ}(M/\mathcal{F})$. 
\end{example}

The relation of the new $K$-theory model to index theory will be discussed in Section \ref{sec:indexcalculations}, where
Corollary \ref{cor:twistingbyasymptoticallyleafwisebundles} gives a formula for indices of longitudinally elliptic operators  twisted by vector bundles $F\to M$ whose classes lie in the image of $p^*:K_{FJ}^0(M/\mathcal{F})\to K^0(M)$.
The following definition provides an analytic way to verify this property. 
Examples of such bundles are obtained by smoothing the construction in the previous example.
\begin{mdef}\label{asymptoticallyleafwisebundles}
Let $F\to M$ be a smooth vector bundle. We say that it is \emph{asymptotically a bundle over the leaf space} if there is a  smooth family of projections $(P_t)_{t\geq 0}\in C^\infty(M;\K) $ such that $F$ is represented by the projection $P_0$ and the norms 
\[\|dP_t|_{T\mathcal{F}}\|=\sup_{0\not=X\in T\mathcal{F}}\frac{\|dP_t(X)\|}{\|X\|}\]
converge to zero for $t\to\infty$.
\end{mdef}
Here, the norm of $dP_t(X)$ is calculated in the $C^*$-algebra $C(M;\K)$.

By reparametrising the $t$-parameter, we can always achieve that $\|\frac{\partial P_t}{\partial t}\|$ and $\frac1{\sqrt{1+t^2}}\|dP_t|_{N\mathcal{F}}\|$ converge to zero for $t\to\infty$, too. Choose a monotonously decreasing function $K:[0,\infty)\to [0,\infty)$ converging to zero at infinity such that
\[\forall t:\,K(t)\geq\max\left(\|dP_t|_{T\mathcal{F}}\|,\,\frac1{\sqrt{1+t^2}}\|dP_t|_{N\mathcal{F}}\|,\,\left\|\frac{\partial P_t}{\partial t}\right\|\right)\,.\]

The $P_t$ compose to give a projection $P\in C_b(\FC;\K)$. If $\gamma:[0,1]\to\FC$ is a smooth path with $t$-component bigger then some fixed $T$, then we decompose $\gamma'=v_L+v_N+\lambda\frac{\partial}{\partial t}$ into longitudinal, normal and $\partial/\partial t$-component and calculate
\begin{align*}
\|P(\gamma(1))-P(\gamma(0))\|&\leq\int_0^1\|(P\circ\gamma)'(\tau)\|d\tau
\\&\leq\int_0^1\left\|dP(v_L(\tau))+dP(v_N(\tau))+\lambda(\tau)\frac{\partial P_t}{\partial t}\right\|d\tau
\\&\leq K(T)\cdot\int_0^1\left(\|v_L(\tau)\|+\sqrt{1+t(\tau)^2}\cdot\|v_N(\tau)\|+|\lambda(\tau)|\right)d\tau
\\&\leq 3K(T)\cdot \int_0^1\|\gamma'(\tau)\|d\tau=3K(T)\cdot L(\gamma)
\end{align*}
This calculation shows that $P$ has vanishing variation, thus 
$P\in\cbar(\FC)$.
Let $x_F$ be its class in $K^0_{FJ}(M/\mathcal{F})$. Formula \eqref{eq:pstar} now immediately implies  
$[F]=p^*(x_F)$. 

Note that the property of being asymptotically a bundle over the leaf space only implies the existence of such an $x_F$, not its uniqueness. Indeed, a different choice of the family of projections $(P_t)_{t\geq 0}$ might yield a different $x_F$.

The element $x_F$ will become important in Corollary \ref{cor:twistingbyasymptoticallyleafwisebundles}.
It should be pointed out that the index $\ind(D_F)$, which is computed in this corollary, does \emph{not} depend on the above-mentioned  choice of the $x_F$ and the family $(P_t)_{t\geq0}$.

\section{Connes' foliation algebra}\label{sec:foliationalgebras}
We briefly recall the construction of Connes' foliation algebra $\FA$.
General references for this section are \cite[Sections 5,6]{ConSurvey}, \cite[Section 2.8]{ConNCG} or \cite[Section 5]{Kordyukov}.

Instead of working with half densities as in \cite{ConSurvey,ConNCG}, we fix once and for all a smooth, positive leafwise $1$-density $\alpha\in C^{\infty}(M,|T\mathcal{F}|)$. It pulls back to smooth densities $r^*\alpha$ on $G_x$ and $s^*\alpha$ on $G^x$ for all $x\in M$ and we will always use these densities for integration. In particular, if $\gamma\in G$ with $x=s(\gamma), y=r(\gamma)$ and $f,g$ are functions on $G^y,G_x$, respectively, then we shall write
\begin{align*}
\int_{\gamma_1\gamma_2=\gamma}f(\gamma_1)g(\gamma_2):&=\int_{\gamma_1\in G^y}f(\gamma_1)g(\gamma_1^{-1}\gamma)s^*\alpha(\gamma_1)
=\int_{\gamma_2\in G_x}f(\gamma\gamma_2^{-1})g(\gamma_2)r^*\alpha(\gamma_2)\,.
\end{align*}

In case $G$ is Hausdorff, the leafwise convolution product
\[(f*g)(\gamma)=\int_{\gamma_1\gamma_2=\gamma} f(\gamma_1)g(\gamma_2)\]
and the involution
\[f^*(\gamma)=\overline{f(\gamma^{-1})}\]
turn the vector space $C_c^\infty(G)$ of smooth complex valued functions with compact support on $G$ into a complex $*$-algebra.

If, however, the manifold structure on $G$ is non-Hausdorff, then $C_c^\infty(G)$ is by definition the vector space of complex functions on $G$ which are finite sums of smooth functions with compact support in some coordinate patch of $G$.
In this case, the convolution product of two functions in $C_c^\infty(G)$ is again in $C_c^\infty(G)$, so $C_c^\infty(G)$ is a complex $*$-algebra in the non-Hausdorff case, too.
This technicality does not interfere with our arguments at all, because we can always assume without loss of generality that our functions are compactly supported in coordinate patches.

For each $x\in M$, the Hilbert space $L^2(G_x)$ is defined by means of the density $r^*\alpha$ on $G_x$.
There is a representation 
$\pi_x:\,C_c^{\infty}(G)\to\mathfrak{B}(L^2(G_x))$
 given by
\[(\pi_x(f)\xi)(\gamma)= \int_{\gamma_1\gamma_2=\gamma}f(\gamma_1)\xi(\gamma_2)\]
for all $f\in C_c^\infty(G)$, $\xi\in L^2(G_x)$ and $\gamma\in G_x$.
\begin{mdef}
The reduced foliation algebra $\FA$ is defined as the completion of $C_c^\infty(G)$ in the pre-$C^*$-norm given by $\|f\|_r=\sup_{x\in M}\|\pi_x(f)\|$.
\end{mdef}

\begin{rem}
All the constructions above work equally well and give the same results if we use continuous instead of smooth functions everywhere.
Note, however, that in this context the definition of continuous functions on a non-Hausdorff $G$ has to be adapted analogously. 
\end{rem}

\begin{mdef}[(\cite{ConSurvey})]
Connes' $K$-theory model for the leaf space of the foliation $\fol$ is
\[K_C^*(M/\mathcal{F}):=K_{-*}(\FA).\]
\end{mdef}

The reduced foliation algebra $\FA$ can be understood as a sub-$C^*$-algebra of $\B\left(\bigoplus_{x\in M}L^2(G_x)\right)$.
We denote the canonical faithful representation by 
\[\pi=\bigoplus_{x\in M}\pi_x:\,\FA\to\B\left(\bigoplus_{x\in M}L^2(G_x)\right)\]
There is also a canonical faithful representation  
\[\tau=\bigoplus_{x\in M}\tau_x:\,C(M)\to\B\left(\bigoplus_{x\in M}L^2(G_x)\right)\]
given by $\tau_x(g)\xi:=r^*g\cdot \xi$.
For $f\in C_c(G)$ and $g\in C(M)$, the pointwise products $r^*g\cdot f$, $s^*g\cdot f$ lie in $C_c(G)$ and
\begin{equation}\label{eq:pre:multiplierrepresentations}
\tau(g)\pi(f)=\pi(r^*g\cdot f),\quad\pi(f)\tau(g)=\pi(s^*g\cdot f).
\end{equation}

\begin{lem}\label{lem:FA:subofmultiplier}
$C(M)$ is canonically a sub-$C^*$-algebra of the multiplier algebra $\mathcal{M}(\FA)$ of $\FA$.
For $f\in \FA$ and $g\in C(M)$ we have
\begin{equation}\label{eq:FA:multiplierrepresentations}
\tau(g)\pi(f)=\pi(gf),\quad \pi(f)\tau(g)=\pi(fg).
\end{equation}
Furthermore, if $g\in C_c(G)$ then
\begin{equation}\label{eq:FA:multiplierforcontfus}
gf=r^*g\cdot f,\quad fg=s^*g\cdot f.
\end{equation}
\end{lem}
\begin{proof}
Formula \eqref{eq:pre:multiplierrepresentations} implies that the image of $\tau$ lies in the largest sub-$C^*$-algebra 
\[D\subset \B\left(\bigoplus_{x\in M}L^2(G_x)\right)\]
which contains the image of $\pi$ as an (essential) ideal. Thus, there is a canonical isometric $*$-homomorphism
\[C(M)\to D\to \mathcal{M}(\FA).\]
Equations \eqref{eq:FA:multiplierrepresentations} and \eqref{eq:FA:multiplierforcontfus} are clear by definition.
\end{proof}

\section{Hilbert modules associated to vector bundles}\label{sec:modules}
Smooth $\Z_2$-graded hermitian vector bundles $E\to M$ give rise to $\Z_2$-graded Hilbert modules over $\FA$ which are particularly important in index theory. We review this construction as presented in \cite[Section 5.3]{Kordyukov}.
 For an introduction into the theory of Hilbert modules we refer to \cite{Lance}.

As this is the first section featuring $\Z_2$-gradings, we take the opportunity to fix some notation:
If $\cE$ is an ungraded Hilbert module over an ungraded $C^*$-algebra $A$ and $I,J\in\N$, then we denote 
$\cE^{I,J}=\cE^I\oplus\cE^J,$
where by definition the first summand is the even graded and the second summand is the odd graded part. In particular this applies to the cases $A=\C$, where $\cE$ is a Hilbert space, and $\cE=A$ being a $C^*$-algebra. For any $\Z_2$-graded Hilbert module $\cE$ we denote by $\B(\cE)$ and $\K(\cE)$ the $C^*$-algebras of adjointable respectively compact operators on $\cE$ equipped with the canonical $\Z_2$-grading. In the special case $\cE=A^{I,J}$ we obtain $\MA_{I,J}(A):=\K(A^{I,J})$, which is the $C^*$-algebra of $(I+J)\times(I+J)$ matrices over $A$ where the diagonal $I\times I$- and $J\times J$-blocks constitute the even part and the off-diagonal blocks are the odd part.

The symbol  $\hatotimes $ will always denote graded tensor products. More specifically, we use it in the context of $\Z_2$-graded $C^*$-algebras for the \emph{maximal} graded tensor product. The minimal graded tensor product will be denoted by $\hatotimes _{\min}$.

Finally, all types of morphisms between $\Z_2$-graded objects are always assumed to be grading preserving, even without explicit mention.

Back to foliations: to define the Hilbert module $\cE$ associated to $E$, let
$\cE^\infty:=C_c^\infty(G,r^*E)$ be the vector space of smooth, compactly supported sections of the bundle $r^*E\to G$. Again, if $G$ is non-Hausdorff, we define it by summing up smooth sections compactly supported in coordinate patches of $G$. There is a right module structure of $\cE^\infty$ over $C_c^\infty(G)$ by letting $f\in C_c^\infty(G)$ act on $\xi\in \cE^\infty$ by the formula
\[(\xi*f)(\gamma) :=\int_{\gamma_1\gamma_2=\gamma} \xi(\gamma_1)f(\gamma_2)\quad\forall\gamma\in G\,.\]
A $C_c^\infty(G)\subset \FA$-valued inner product $\langle\,\_\,,\,\_\,\rangle_{\cE^\infty}$ on $\cE^\infty$ is defined by
\[\langle\xi,\zeta\rangle_{\cE^\infty}(\gamma) :=\int_{\gamma_1\gamma_2=\gamma} \langle\xi(\gamma_1^{-1}),\zeta(\gamma_2)\rangle_{E_{r(\gamma_2)}}\,.\]
This inner product is positive and defines a norm $\|\xi\|_r:=\|\langle\xi,\xi\rangle_{\cE^\infty}\|_r^{1/2}$ on $\cE^\infty$.
\begin{mdef}
The $\Z_2$-graded Hilbert module $\cE$ associated to the hermitian vector bundle $E\to M$ is defined as the completion of $\cE^\infty$ in the norm $\|\_\|_r$.
The module multiplication of $\FA$ on $\cE$ and the $\FA$-valued inner product $\langle\,\_\,,\,\_\,\rangle_{\cE}$ on $\cE$ are defined by extending module multiplication of $C_c^\infty(G)$ on $\cE^\infty$ and $C_c^\infty(G)$-valued inner product on $\cE^\infty$ continuously.
\end{mdef}

Again, one can perform these constructions using continuous instead of smooth sections.

If $E=\C^{I,J}\times M\to M$ is a trivial, $\Z_2$-graded bundle, then the associated $\Z_2$-graded Hilbert-$\FA$-module is $\cE=(\FA)^{I,J}.$
Its $\Z_2$-graded $C^*$-algebras of compact and adjointable operators are $\K(\cE)=\MA_{I,J}(\FA)$ and $\B(\cE)=\MA_{I,J}(\mathcal{M}(\FA))$, respectively.

An arbitrary $\Z_2$-graded vector bundle $E\to M$ may be embedded (grading preservingly) into a trivial bundle $\C^{I,J}\times M\to M$, such that $E$ is the image of a projection $p\in \MA_I(C(M))\oplus \MA_J(C(M))\subset \MA_{I,J}(C(M))$. Thus, $p$ may be seen as a projection in 
\[\B((\FA)^{I,J})=\MA_{I,J}(\mathcal{M}(\FA))\]
which we also denote by $p$. It is easy to see, that the $\Z_2$-graded Hilbert-$\FA$-module $\cE$ associated to $E$ is canonically isomorphic to the image of this projection,
$\cE\cong \im(p)\subset (\FA)^{I,J}.$
Consequently, 
\begin{equation}\label{eq:Moritainclusion}
\K(\cE)=p\MA_{I,J}(\FA)p\subset \MA_{I,J}(\FA).
\end{equation}

We will need the following faithful representation of $\K(\cE)$:
\begin{lem}\label{lem:faithfulrepresentationofcompacts}
There is a canonical isometric inclusion 
\[\pi:\,\K(\cE)\xrightarrow{\subset} \mathfrak{B}\left(\bigoplus_{x\in M}L^2(G_x,r^*E)\right)\]
with the following property:
if $T\in \K(\cE)$ is given on $C_c(G,r^*E)$ by convolution with $a\in C_c(G,r^*E\otimes s^*E^*)$, then $\pi(T)$ acts on each summand   $L^2(G_x,r^*E)$ also by convolution with $a$.
\end{lem}
\begin{proof}
Simply compose the inclusion \eqref{eq:Moritainclusion} componentwise with the representation $\pi$ of $\FA$.
Using Lemma \ref{lem:FA:subofmultiplier}, it is straightforward to verify that the image of this composition is in fact contained in 
\[\B\left(\bigoplus_{x\in M}L^2(G_x,r^*E)\right)\subset\B\left(\bigoplus_{x\in M}L^2(G_x,\C^{I,J})\right).\]

The claimed property of this representation of $\K(\cE)$ follows directly from the analogous property of the canonical representation of $\FA$.
\end{proof}

There is also a canonical isometric inclusion  
\[\tau:\,C(M)\xrightarrow{\subset} \B\left(\bigoplus_{x\in M}L^2(G_x,r^*E)\right)\]
where $g\in C(M)$ acts on each $L^2(G_x,r^*E)$ by pointwise multiplication with $r^*g$.
Completely analogous to Lemma \ref{lem:FA:subofmultiplier} we have:

\begin{lem}\label{lem:subofmultiplier}
$C(M)$ is canonically a sub-$C^*$-algebra of the multiplier algebra $\mathcal{M}(\K(\cE))=\B(\cE)$ of $\K(\cE)$.
For $T\in \K(\cE)$ and $g\in C(M)$ we have
\[
\tau(g)\pi(T)=\pi(gT),\quad \pi(T)\tau(g)=\pi(Tg).
\]
Furthermore, if $T\in \K(\cE)$ is given by convolution with $a\in C_c(G,r^*E\otimes s^*E^*)$ and $g\in C(M)$, then $r^*g\cdot a,s^*g\cdot a\in  C_c(G,r^*E\otimes s^*E^*)$, too, and $gT$ is convolution with $r^*g\cdot a$ whereas $Tg$ is convolution with $s^*g\cdot a$.
\end{lem}

\section{The asymptotic category and $E$-theory}\label{sec:Etheory}
This section is a brief summary of the basic definitions and properties of the asymptotic category and $E$-theory.
We use the picture of $E$-theory presented in \cite{HigGue}. 
A more detailed exposition of $E$-theory, which is based on a slightly different definition, is found in \cite{GueHigTro}.

\begin{mdef}[({\cite[Definition 2.2]{HigGue},\cite[Definition 1.1]{GueHigTro}})]
Let $B$ be a $\Z_2$-graded $C^*$-algebra. The \emph{asymptotic $C^*$-algebra of $B$} is 
\[\mathfrak{A}(B):=C_b([1,\infty),B)/C_0([1,\infty),B).\]
$\mathfrak{A}$ is a functor from the category of $\Z_2$-graded $C^*$-algebras into itself.

An \emph{asymptotic morphism} is a graded $*$-homomorphism $A\to \mathfrak{A}(B)$.
\end{mdef}

\begin{mdef}[({\cite[Definition 2.3]{HigGue},\cite[Definition 2.2]{GueHigTro}})]
Let $A,B$ be  $\Z_2$-graded $C^*$-algebras.
The \emph{asymptotic functors} $\mathfrak{A}^0,\mathfrak{A}^1,\dots$ are defined by $\mathfrak{A}^0(B)=B$ and
\[\mathfrak{A}^n(B)=\mathfrak{A}(\mathfrak{A}^{n-1}(B)).\]
Two $*$-homomorphisms $\phi^0,\phi^1:A\to\mathfrak{A}^n(B)$ are $n$-homotopic if there exists a $*$-homomorphism
$\Phi:A\to\mathfrak{A}^n(B[0,1])$, called $n$-homotopy between $\phi^0,\phi^1$, from which the $*$-homomorphisms $\phi^0,\phi^1$ are recovered as the compositions
\[A\xrightarrow{\Phi}\mathfrak{A}^n(B[0,1]) \xrightarrow{\text{evaluation at }0,1}\mathfrak{A}^n(B).\]
\end{mdef}

\begin{lem}[({\cite[Proposition 2.3]{GueHigTro}})]
The relation of $n$-homotopy is an equivalence relation on the set of $*$-homomorphisms from $A$ to $\mathfrak{A}^n(B)$.
\end{lem}

\begin{mdef}[({\cite[Definition 2.4]{HigGue},\cite[Definition 2.6]{GueHigTro}})]
Let $A,B$  be  $\Z_2$-graded $C^*$-algebras. Denote by $\llbracket A,B\rrbracket_n$ the set of $n$-homotopy classes of $*$-homomorphisms  from $A$ to $\mathfrak{A}^n(B)$.
\end{mdef}

There are two natural transformations $\mathfrak{A}^n\to\mathfrak{A}^{n+1}$:
The first is defined by including $\mathfrak{A}^n(B)$ into $\mathfrak{A}^{n+1}(B)=\mathfrak{A}(\mathfrak{A}^n(B))$ as constant functions.
The second is defined by applying the functor $\mathfrak{A}^n$ to the inclusion of $B$ into $\mathfrak{A}B$ as constant functions.
Both of them define maps 
$\llbracket A,B\rrbracket_n\to\llbracket A,B\rrbracket_{n+1}.$
\begin{lem}[({\cite[Proposition 2.8]{GueHigTro}})]
The above natural transformations define the same map $\llbracket A,B\rrbracket_n\to\llbracket A,B\rrbracket_{n+1}.$
\end{lem}
These maps organize the sets $\llbracket A,B\rrbracket_n$ into a directed system
\[\llbracket A,B\rrbracket_0\to\llbracket A,B\rrbracket_1\to\llbracket A,B\rrbracket_2\to\dots\]
\begin{mdef}[({\cite[Definition 2.5]{HigGue},\cite[Definition 2.7]{GueHigTro}})]
Let  $A,B$  be  $\Z_2$-graded $C^*$-algebras. Denote by $\llbracket A,B\rrbracket_\infty$ the direct limit of the above directed system.
We denote the class of a $*$-homomorphism $\phi:A\to\mathfrak{A}^n(B)$ by $\llbracket\phi\rrbracket$.
\end{mdef}

\begin{prop}[({\cite[Proposition 2.12]{GueHigTro}})]\label{prop:asymptoticcomposition}
Let $\phi:A\to\mathfrak{A}^n(B)$ and $\psi:B\to\mathfrak{A}^m(C)$ be $*$-homomorphisms. The class $\llbracket\psi\rrbracket\circ\llbracket\phi\rrbracket\in \llbracket A,C\rrbracket_\infty$ of the composite $*$-homomorphism
\[A\xrightarrow{\phi}\mathfrak{A}^n(B) \xrightarrow{\mathfrak{A}^n(\psi)} \mathfrak{A}^{n+m}(C)\]
depends only on the classes $\llbracket\phi\rrbracket\in\llbracket A,B\rrbracket_\infty$, $\llbracket\psi\rrbracket\in\llbracket B,C\rrbracket_\infty$ of $\phi,\psi$. The composition law
\[\llbracket A,B\rrbracket_\infty\times\llbracket B,C\rrbracket_\infty\to \llbracket A,C\rrbracket_\infty,\quad(\llbracket\phi\rrbracket,\llbracket\psi\rrbracket)\mapsto \llbracket\psi\rrbracket\circ\llbracket\phi\rrbracket\]
so defined is associative.
\end{prop}
For example, if $n=m=1$ and $\phi,\psi$ lift to continuous maps 
\begin{align*}
\tilde\phi:A&\to C_b([1,\infty),B),&a&\mapsto[t\mapsto\tilde\phi_t(a)],
\\\tilde\psi:B&\to C_b([1,\infty),C),&b&\mapsto[s\mapsto\tilde\psi_s(b)],
\end{align*}
respectively, then $\llbracket\psi\rrbracket\circ\llbracket\phi\rrbracket$ is represented by
\[A\to\mathfrak{A}^2(C),\quad a\mapsto\overline{t\mapsto\overline{s\mapsto \tilde\psi_s(\tilde\phi_t(a))}},\]
where the overline denotes equivalence classes.
We will make use of this formula a few times later on.

According to the proposition, we obtain a category:
\begin{mdef}[({\cite[Definition 2.6]{HigGue},\cite[Definition 2.13]{GueHigTro}})]
The \emph{asymptotic category} is the category whose objects are $\Z_2$-graded $C^*$-algebras, whose morphisms are elements of the sets $\llbracket A,B\rrbracket_\infty$, and whose composition law is defined in Proposition \ref{prop:asymptoticcomposition}.
\end{mdef}
The identity morphism $1_A\in \llbracket A,A\rrbracket_\infty$ is represented by the identity $\id_A:A\to A=\mathfrak{A}^0(A)$.

For arbitrary $\Z_2$-graded $C^*$-algebras $B,D$, there are canonical asymptotic morphisms
\begin{align*}
\mathfrak{A}(B)\hatotimes  D&\to \mathfrak{A}(B\hatotimes  D)&\bar g\hatotimes  d&\mapsto\overline{t\mapsto g(t)\hatotimes  d}
\\D\hatotimes \mathfrak{A}(B)&\to \mathfrak{A}(D\hatotimes  B)&d\hatotimes \bar g&\mapsto\overline{t\mapsto d\hatotimes  g(t)}
\end{align*}
and inductively also canonical $*$-homomorphisms $\mathfrak{A}^n(B)\hatotimes  D\to \mathfrak{A}^n(B\hatotimes  D)$, $D\hatotimes \mathfrak{A}^n(B)\to \mathfrak{A}^n(D\hatotimes  B)$. This is a consequence of {\cite[Lemmas 4.1, 4.2 \& Chapter 3]{GueHigTro}.

\begin{prop}[({\cite[Theorem 4.6]{GueHigTro}})]
The asymptotic category is a mo\-no\-idal category with respect to the \emph{maximal graded tensor product} $\hatotimes $ of $C^*$-algebras and a tensor product on the morphism sets,
\[\hatotimes :\llbracket A_1,B_1\rrbracket_\infty\times\llbracket A_2,B_2\rrbracket_\infty\to\llbracket A_1\hatotimes  A_2,B_1\hatotimes  B_2\rrbracket_\infty,\]
with the following property:
If 
$\llbracket\phi\rrbracket\in \llbracket A_1,B_1\rrbracket_\infty$ and $\llbracket\psi\rrbracket\in \llbracket A_2,B_2\rrbracket_\infty$
are represented by $\phi:A_1\to\mathfrak{A}^m(B_1)$ and $\psi:A_2\to\mathfrak{A}^n(B_2)$ , respectively, and $D$ is another $\Z_2$-graded $C^*$-algebra, then 
\begin{align*}
\llbracket\phi\rrbracket\hatotimes 1_D&\in \llbracket A_1\hatotimes  D,B_1\hatotimes  D\rrbracket_\infty,
\\1_D\hatotimes  \llbracket\psi\rrbracket&\in \llbracket D\hatotimes  A_2,D\hatotimes  B_2\rrbracket_\infty
\end{align*}
are represented by the compositions
\begin{align*}
A_1\hatotimes  D&\xrightarrow{\phi\hatotimes \id_D}\mathfrak{A}^m(B_1)\hatotimes  D\to\mathfrak{A}^m(B_1\hatotimes  D),
\\D\hatotimes  A_2&\xrightarrow{\id_D\hatotimes \psi}D\hatotimes \mathfrak{A}^n(B_2)\to\mathfrak{A}^n(D\hatotimes  B_2),
\end{align*}
respectively. 
\end{prop}
The general form of the tensor product is of course
\[\llbracket\phi\rrbracket\hatotimes \llbracket\psi\rrbracket=(\llbracket\phi\rrbracket\hatotimes 1_{B_2})\circ(1_{A_1}\hatotimes \llbracket\psi\rrbracket)=(1_{B_1}\hatotimes \llbracket\psi\rrbracket)\circ(\llbracket\phi\rrbracket\hatotimes 1_{A_2}).
\]

There is an obvious monoidal functor from the category of $\Z_2$-graded $C^*$-algebras into the asymptotic category which is the identity on the objects and 
maps a $*$-homomorphism $A\to B$ to its class in $\llbracket A,B\rrbracket_\infty$ by considering it as a $*$-homomorphism $A\to \mathfrak{A}^0(B)$.

The definition of $E$-theory involves the following two $\Z_2$-graded $C^*$-algebas. The first is $\widehat\K=\B(\widehat\ell^2)=\MA_{1,1}(\K)$ -- the $\Z_2$-graded $C^*$-algebra of compact operators on the $\Z_2$-graded Hilbert space $\widehat\ell^2=\ell^2\oplus\ell^2$ with even and odd part equal to the standard  separable, infinite dimensional Hilbert space $\ell^2$.

The role of $\widehat\K$ is stabilization: Given two separable, $\Z_2$-graded Hilbert spaces $H_1,H_2$, any isometry $V:H_1\hatotimes \widehat\ell^2\to H_2\hatotimes \widehat\ell^2$ defines an injective $*$-homomorphism
\[\operatorname{Ad}_V:\,\K(H_1)\hatotimes \widehat\K\to\K(H_2)\hatotimes  \widehat\K,\quad T\mapsto VTV^*.\]
The homotopy class of $\operatorname{Ad}_V$ is independent of the choice of $V$ and therefore defines a canonical isomorphism between $\K(H_1)\hatotimes \widehat\K$ and $\K(H_2)\hatotimes  \widehat\K$ in the asymptotic category. 
The proof is a standard argument which will also be referred to later on: Any two isometries $V_0,V_1:H_1\hatotimes \widehat\ell^2\to H_2\hatotimes \widehat\ell^2$ are homotopic in the strong operator topology to 
\[\begin{pmatrix}V_0\\0\end{pmatrix}\,,\, \begin{pmatrix}0\\V_1\end{pmatrix}\,:\quad H_1\hatotimes \widehat\ell^2\to 
H_2\hatotimes \widehat\ell^2\oplus H_2\hatotimes \widehat\ell^2\cong H_2\hatotimes \widehat\ell^2\,,\]
respectively, by a Hilbert's hotel argument. These in turn are homotopic to each other via $\tilde V_t:=\begin{pmatrix}\cos(t)V_0\\\sin(t)V_1\end{pmatrix}$.
Combined we obtain a strongly continuous homotopy $V_t$ between $V_0$ and $V_1$. Now for every $T\in\K(H_1)\hatotimes\widehat\K$ the compactness of $T$ implies that $t\mapsto\operatorname{Ad}_{V_t}(T)$ is even \emph{norm-continuous}, i.\,e.\ $\operatorname{Ad}_{V_t}$ is a homotopy of $*$-homomorphisms and the claim follows.

In particular, $\widehat\K\hatotimes \widehat\K$, $\K\hatotimes \widehat\K$ and $\MA_{I,J}(\C)\hatotimes \widehat\K$ are all canonically isomorphic to $\widehat\K$.

The second $\Z_2$-graded $C^*$-algebra is $C_0(\R)$, but with non-trivial grading given by 
the direct sum decomposition into even and odd functions. This $\Z_2$-graded $C^*$-algebra is denoted by $\mathcal{S}$.

Recall from \cite[Section 1.3]{HigGue} that $\mathcal{S}$ is also a co-algebra with co-unit $\eta:\mathcal{S}\to\C,f\mapsto f(0)$ and a co-multiplication $\Delta:\mathcal{S}\to\mathcal{S}\hatotimes \mathcal{S}$. The definition of $\Delta$ is not relevant to us, as we shall explain below. It is enough to know the axioms of a co-algebra, i.\,e.\ that
\begin{equation}\label{eq:coalgebraaxioms}
\xymatrix{
\mathcal{S}\ar[r]^{\Delta}\ar[d]_{\Delta}&\mathcal{S}\hatotimes \mathcal{S}\ar[d]^{\id\hatotimes \Delta}&\mathcal{S}\ar[d]_{\id}\ar[r]^{\id}\ar[dr]^{\Delta}&\mathcal{S}
\\\mathcal{S}\hatotimes \mathcal{S}\ar[r]_{\Delta\hatotimes \id}&\mathcal{S}\hatotimes \mathcal{S}\hatotimes \mathcal{S}&\mathcal{S}&\mathcal{S}\hatotimes \mathcal{S}\ar[l]^{\eta\hatotimes \id}\ar[u]_{\eta\hatotimes \id}
}\end{equation}
commute.

\begin{mdef}\label{def:Etheory}
Let $A,B$ be $\Z_2$-graded $C^*$-algebras. The \emph{$E$-theory} of $A,B$  is 
\[E(A,B)=\llbracket\mathcal{S}\hatotimes  A\hatotimes \widehat\K,B\hatotimes \widehat\K\rrbracket_\infty.\]
It is a group with addition given by direct sum 
 of $*$-homomorphisms 
 \[\mathcal{S}\hatotimes  A\hatotimes \widehat\K\to \mathfrak{A}^n(B\hatotimes \widehat\K)\]
  (via an inclusion $\widehat\K\oplus\widehat\K\hookrightarrow\widehat\K$, which is canonical up to homotopy) and the zero element represented by the zero $*$-homomorphism.
\end{mdef}

\begin{rem}
By \cite[Theorem 2.16]{GueHigTro}, this definition is equivalent to \cite[Definition 2.1]{HigGue} when $A,B$ are separable.
For non-separable $C^*$-al\-ge\-bras, however, it is essential to use Definition \ref{def:Etheory}, because otherwise the products defined below might not exist.
\end{rem}

There is a composition product 
\[E(A,B)\otimes E(B,C)\to E(A,C),\quad (\phi,\psi)\mapsto \psi\circ\phi,\]
where $\psi\circ\phi\in E(A,C)$ is defined to be the composition 
\[\mathcal{S}\hatotimes  A\hatotimes \widehat\K\xrightarrow{\Delta\hatotimes \id_{A\hatotimes \widehat\K}}\mathcal{S}\hatotimes \mathcal{S}\hatotimes  A\hatotimes \widehat\K\xrightarrow{\id_{\mathcal{S}}\hatotimes \phi}\mathcal{S}\hatotimes  B\hatotimes \widehat\K\xrightarrow{\psi}C\hatotimes \widehat\K\]
of morphisms in the asymptotic category.

There is also an exterior product
\[E(A_1,B_1)\otimes E(A_2,B_2)\to E(A_1\hatotimes  A_2,B_1\hatotimes  B_2),\quad (\phi,\psi)\mapsto \phi\hatotimes \psi,\]
where $\phi\hatotimes \psi\in  E(A_1\hatotimes  A_2,B_1\hatotimes  B_2)$ is defined to be the composition
\begin{align*}
\mathcal{S}\hatotimes  A_1\hatotimes  A_2\hatotimes \widehat\K&\xrightarrow{\Delta\hatotimes \id}\mathcal{S}\hatotimes \mathcal{S}\hatotimes  A_1\hatotimes  A_2\hatotimes \widehat\K
\cong\mathcal{S}\hatotimes  A_1\hatotimes \widehat\K\hatotimes \mathcal{S}\hatotimes  A_2\hatotimes \widehat\K
\\&\xrightarrow{\phi\hatotimes \psi}B_1\hatotimes \widehat\K\hatotimes  B_2\hatotimes \widehat\K
\cong B_1\hatotimes  B_2\hatotimes \widehat\K
\end{align*}
of morphisms in the asymptotic category.

\begin{thm}[({\cite[Theorems 2.3, 2.4]{HigGue}})]
With these composition and exterior products,
the $E$-theory groups $E(A,B)$ are the morphism groups in an additive monoidal category $\mathbf{E}$ whose objects are the $\Z_2$-graded $C^*$-algebras. 
\end{thm}
We conclude this section by mentioning some properies of $E$-theory needed for our computations. Our earlier observations imply:

\begin{thm}[(Stability)]\label{thm:EStability}
For any separable $\Z_2$-graded Hilbert space $H$, the $\Z_2$-graded $C^*$-algebra $\K(H)$ is canonically isomorphic in the category $\mathbf{E}$ to $\C$. In particular, this applies to $\widehat\K$, $\K$ and $\MA_{I,J}(\C)$.
\end{thm}

\begin{thm}[({\cite[Theorems 2.3, 2.4]{HigGue}})]\label{thm:EFunctoriality}
There is a monoidal functor from the asymptotic category into $\mathbf{E}$ which is the identity on the objects and maps $\phi\in\llbracket A,B\rrbracket_\infty$ to the morphism
\[\mathcal{S}\hatotimes  A\hatotimes \widehat\K\xrightarrow{\llbracket\eta\rrbracket\hatotimes \phi\hatotimes 1_{\widehat\K}}B\hatotimes \widehat\K\]
in the asymptotic category, which we denote by the same letter $\phi$.
\end{thm}

Thus, by taking the $E$-theory product with this $E$-theory element, we obtain homomorphisms 
\[E(D,A)\xrightarrow{\phi\circ}E(D,B),\quad E(B,D)\xrightarrow{\circ\phi}E(A,D)\]
for any third $\Z_2$-graded $C^*$-algebra $D$.
Consequently, the $E$-theory groups are contravariantly functorial in the first variable and  covariantly functorial in the second variable with respect to morphisms in the asymptotic category and in particular with respect to $*$-homomorphisms.

These functorialities can be computed more easily than arbitrary composition products in $E$-theory:
If $\psi\in E(D,A)$ and $\phi\in \llbracket A,B\rrbracket_\infty$, then $\phi\circ\psi\in E(D,B)$ is the composition
\[\mathcal{S}\hatotimes  D\hatotimes \K\xrightarrow{\psi}A\hatotimes \widehat\K\xrightarrow{\phi\hatotimes 1_{\widehat\K}}B\hatotimes \K\]
in the asymptotic category. This is, because the co-multiplication $\Delta:\mathcal{S}\to\mathcal{S}\hatotimes \mathcal{S}$ in the definition of the composition product in $E$-theory cancels with the co-unit $\eta:\mathcal{S}\to\C$ appearing in the functor from the asymptotic category to $E$-theory by \eqref{eq:coalgebraaxioms}.

Similarly, if $\psi\in E(B,D)$ and $\phi\in \llbracket A,B\rrbracket_\infty$, then $\psi\circ\phi\in E(A,D)$ is the composition
\[\mathcal{S}\hatotimes  A\hatotimes \widehat\K\xrightarrow{1_{\mathcal{S}}\hatotimes \phi\hatotimes  1_{\widehat\K}}\mathcal{S}\hatotimes  B\hatotimes \widehat\K\xrightarrow{\psi}B\hatotimes \widehat\K,\]
and the exterior product of $\phi\in E(A_1,B_1)$ and $\psi\in \llbracket A_2,B_2\rrbracket_\infty$ is the composition
\[\mathcal{S}\hatotimes  A_1\hatotimes  A_2\hatotimes \widehat\K\xrightarrow{\phi\hatotimes \psi}B_1\hatotimes  B_2\hatotimes \widehat\K.\]

In our applications in the following sections, we have to compute  products only  in cases where one of the factors comes from an asymptotic morphism and not from an $E$-theory class. This is the reason why our computations will not involve $\Delta$ and thus we don't have to know its definition.

Generalizing the functor from the asymptotic category to the $E$-theory category, 
elements of $E(A,B)$ are also obtained from any morphism in the asymptotic category of the form
\[A\hatotimes \K(H_1)\to B\hatotimes \K(H_2)\quad\text{or}\quad\mathcal{S}\hatotimes  A\hatotimes \K(H_1)\to B\hatotimes \K(H_2)\]
where $H_1,H_2$ are arbitrary separable, $\Z_2$-graded Hilbert spaces.
The $E$-theory element is obtained by tensoring with 
$\llbracket\eta\rrbracket\hatotimes \id_{\widehat\K}$ respectively $\id_{\widehat\K}$ and applying stability.

We shall also need invariance under Morita equivalence, which is the second part of the following theorem:
\begin{thm}\label{thm:Morita}
Let $\cE$ be a countably generated, $\Z_2$-graded Hilbert module over a $\Z_2$-graded $C^*$-algebra $B$. Given an isometric, grading preserving inclusion $V:\cE\subset B\hatotimes  H$, where $H$ is a separable, $\Z_2$-graded Hilbert space, we obtain an isometric $*$-homomorphism $\operatorname{Ad}_V:\K(\cE)\subset B\hatotimes \K(H)$ which induces an element
$\Theta_\cE\in E(\K(\cE),B)$.
\begin{enumerate}
\item This element $\Theta_\cE$ always exists and is independent of the choice of the inclusion.
\item If $\cE$ is full, i.\,e.\ $\langle\cE,\cE\rangle=B$, and $B$ has a strictly positive element, then $\Theta_\cE$ is invertible.
\end{enumerate}
\end{thm}
\begin{proof}
Because of stability we may assume $H=\widehat\ell^2$.
The existence of an inclusion $\cE\subset B\hatotimes \widehat\ell^2$ is guaranteed by Kasparov's stabilization Theorem \cite[Theorem 2]{KasparovHilbertModules}. 

To prove uniqueness, we use the standard argument seen earlier: Any two isometric and grading preserving inclusions $V_{0,1}:\cE\xrightarrow{\subset} B\hatotimes \widehat\ell^2$ are homotopic in the strong operator topology and thus $\operatorname{Ad}_{V_{0,1}}$ are homotopic $*$-homomorphisms between the $C^*$-algebras of compact operators.
\iffalse
To prove uniqueness, let $V_{0,1}:\cE\to B\hatotimes \widehat\ell^2$ be two isometric, grading preserving inclusions.
A standard argument shows that all inclusions
$\cE\subset B\hatotimes \widehat\ell^2$ -- including $V\hatotimes \id_{\widehat\ell^2}$ -- are homotopic in the strong operator topology.
In more detail: by a Hilbert's hotel-argument, the inclusions $\widehat\ell^2\xrightarrow{\subset} \widehat\ell^2\oplus \widehat\ell^2\cong \widehat\ell^2$ as first and second summand are strongly homotopic to the identity. 
This gives rise to strong homotopies between $V_{0,1}$ and 
$\begin{pmatrix}V_0\\0\end{pmatrix}\,,\,\begin{pmatrix}0\\V_1\end{pmatrix}:
\,\cE\to B\hatotimes \widehat\ell^2\oplus B\hatotimes \widehat\ell^2\cong B\hatotimes \widehat\ell^2\,,
$
respectively, which in turn are homotopic to each other via the homotopy $V_t:=\begin{pmatrix}\cos(t)V_0\\\sin(t)V_1\end{pmatrix}$.

As conjugating \emph{compact} operators by strong homotopies yields a \emph{norm-continuous} path between operators, we see that the $*$-homomorphisms $\operatorname{Ad}_{V_0,1}$ defined on $\K(\cE)$ are homotopic to each other.
This proves uniqueness.
\fi

The same argument also proves that we can homotop $V\hatotimes \id_{\widehat\ell^2}$ to the isomorphism $\cE\otimes\widehat\ell^2\cong B\hatotimes \widehat\ell^2$ which always exists under the additional assumptions of the second part by \cite[Theorems 1.9]{MingoPhillips}.
Thus, $\Theta_\cE$ is represented by a $*$-isomorphism and is therefore invertible.

Note that the cited theorems are only formulated for the ungraded case, but $\Z_2$-gradings are readily implemented into their proofs.
\end{proof}

\begin{cor}\label{cor:bundlemorita}
If $E\to M$ is nowhere zero dimensional, then the inclusion 
\[\K(\cE)\subset \MA_{I,J}(\FA)\]
of Equation 
\eqref{eq:Moritainclusion} induces an invertible element of
$E(\K(\cE),\FA).$
\end{cor}
\begin{proof}
The foliation algebra $\FA$ is separable and thus contains a strictly positive element by \cite{AarnesKadison}. 
The Hilbert-$\FA$-module $\cE$ is full, because $E$ is nowhere zero dimensional.
\end{proof}

The additive monoidal category $\mathbf{KK}$, whose objects are separable $\Z_2$-graded $C^*$-algebras, has similar properties as $\mathbf{E}$ (\cite{KasparovOperatorKFunctor}, see also \cite{BlaKK}).
Recall that its morphism groups $KK(A,B)$ are defined for all $\Z_2$-graded $C^*$-algebras where $A$ is separable ($B$ need not be separable).
The functor from the category of $\Z_2$-graded separable $C^*$-algebras and $*$-homomorphisms to $\mathbf{E}$ factors canonically through $\mathbf{KK}$. 
The maps $KK(A,B)\to E(A,B)$ of this functor also exist when $B$ is not separable and are  isomorphisms if $A$ is nuclear. 
In particular, the functor $E(\C,\,\_\,)$ from the category of  $\Z_2$-graded $C^*$-algebras to abelian groups is canonically naturally isomorphic to $K$-theory.

Index theory is usually formulated in terms of $KK$-theory whereas we have to use $E$-theory.
Therefore, we have to know these maps explicitly to transfer basic notions to $E$-theory.

In the unbounded picture of $KK$-theory of \cite{BaaJul} (see also \cite[Section 17.11]{BlaKK}), elements of $KK(A,B)$ are represented by triples $(\cE,\rho,D)$, where
\begin{itemize}
\item $\cE$ is a countably generated, $\Z_2$-graded Hilbert-$B$-module,
\item $\rho:A\to\B(\cE)$ is a grading preserving representation of $A$ on $\cE$,
\item $D$ is an odd selfadjoint regular operator on $\cE$
\end{itemize}
such that for all $a$ in a dense subset of $A$, the commutator $[\rho(a),D]$ is densely defined and extends to a bounded operator on $\cE$ and $\rho(a)(D\pm i)^{-1}\in\K(\cE)$.

\begin{prop}[({cf.\ \cite[Section 8]{ConHig}})]\label{prop:KKEiso}
Under the canonical map
\[KK(A,B)\to E(A,B),\]
the element represented by the triple $(\cE,\rho,D)$ is mapped to the class of the asymptotic morphism
\[\mathcal{S}\hatotimes  A\to \mathfrak{A}(\K(\cE)),\quad f\hatotimes  a\mapsto \overline{t\mapsto \rho(a)f(t^{-1}D)}\]
in $E(A,\K(\cE))$ composed with the element $\Theta_\cE\in E(\K(\cE),B)$ of Theorem \ref{thm:Morita}.
\end{prop}

\section[The module structure]{The module structure on Connes' $K$-theory model}\label{sec:modulestructure}
We are now going to define the module structure. Recall the objects we have introduced so far: $\fol$ is a foliation and $\FC$ its foliated cone constructed with respect to some Riemannian metric on $M$. Furthermore, $E$ denotes a smooth hermitian vector bundle over $M$ and $\cE$ the associated Hilbert module.

If $D$ is any coefficient $C^*$-algebra and $g\in \kbar(\FC,D)$, then $g_t\in C(M)\hatotimes  D$ denotes the restriction of $g$ to $\{t\}\times M\subset\FC$. Furthermore, we may consider $g_t$ as an element of $\B(\cE)\hatotimes  D$ by Lemma \ref{lem:subofmultiplier}.

The main ingredient of the module structure 
 is the following asymptotic morphism.
\begin{thm}
For each coefficient $C^*$-algebra $D$, there is an asymptotic morphism
\begin{align*}
\mathfrak{m}_D:\;\kfrak(\FC,D)\hatotimes \K(\cE)&\to\mathfrak{A}(\K(\cE)\hatotimes_{\min} D)
\\\bar g\hatotimes  T&\mapsto \overline{t\mapsto g_t\cdot(T\hatotimes  1_{\tilde D})}.
\end{align*}
\end{thm}
\begin{proof}
Denote by $\tilde D$ the unitalization of $D$.
There is an obvious inclusion 
\[\alpha_{\max}:\,\K(\cE)\to\mathfrak{A}(\B(\cE)\hatotimes \tilde D)\]
as constant functions.
Furthermore,  the composition
\begin{align*}
\kbar(\FC,D)&\subset C_b([0,\infty)\times M,D)=C_b([0,\infty),C(M)\hatotimes  D)
\\&\subset C_b([0,\infty),\B(\cE)\hatotimes \tilde D)
\twoheadrightarrow \mathfrak{A}(\B(\cE)\hatotimes \tilde D)
\end{align*}
obviously descends to give a $*$-homomorphism
\[\beta_{\max}:\,\kfrak(\FC,D)\to\mathfrak{A}(\B(\cE)\hatotimes \tilde D).\]
Let $\alpha,\beta$ be obtained from $\alpha_{\max},\beta_{\max}$ by passing from the maximal tensor product to the minimal tensor product $\B(\cE)\hatotimes_{\min}\tilde D$.

In the following lemma, the vanishing variation of $g$ enters the game:
\begin{lem}\label{asymptoticcommutation}
For all $T\in\K(\cE)$ and $g\in\kbar(\FC,D)$, the commutator $[\beta(\bar g),\alpha(T)]\in\mathfrak{A}(\B(\cE)\hatotimes_{\min}\tilde D)$ vanishes.
\end{lem}
\begin{proof}
Let $\rho:\tilde D\to \B(H')$ be a faithful representation. The minimal tensor product $\K(\cE)\hatotimes _{\min} \tilde D$ may be defined as the image of the $*$-homomorphism
\[\pi\hatotimes \rho:\,\K(\cE)\hatotimes  \tilde D\to\B\left(\bigoplus_{x\in M}L^2(G_x,r^*E)\hatotimes  H'\right).\]
Lemma \ref{lem:subofmultiplier} implies
\begin{align*}
(\pi\hatotimes \rho)(g_t\cdot(T\hatotimes  1_{\tilde D}))&=(\tau\hatotimes \rho)(g_t)\cdot(\pi\hatotimes \rho)(T\hatotimes  1_{\tilde D}),
\\(\pi\hatotimes \rho)((T\hatotimes  1_{\tilde D})\cdot g_t)&=(\pi\hatotimes \rho)(T\hatotimes  1_{\tilde D})\cdot (\tau\hatotimes \rho)(g_t).
\end{align*}
By definition of $\tau$, the operator $(\tau\hatotimes \rho)(g_t)$ acts on each $L^2(G_x,r^*E)\hatotimes  H'\cong L^2(G_x,r^*E\hatotimes  H')$ by multiplication with $r^*(\rho\circ g_t)\in C(G_x,\B(H'))$.

We may assume without loss of generality that the operator $T\in \K(\cE)$ acts on 
$C_c(G,r^*E) \subset \cE$ by convolution with some $a\in C_c(G,r^*E\hatotimes  s^*E^*)$ supported in a compact subset $K$ of some coordinate chart of $G$.
According to Corollary \ref{commutationlimit}, the norms
\[\varepsilon_t:=\left\|(s^*g_t-r^*g_t)|_K\right\|_{C(K,D)}\]
tend to zero as $t\to\infty$. For $\xi\in L^2(G_x,r^*E)\hatotimes  H'\cong L^2(G_x,r^*E\hatotimes  H')$ we may now calculate:
\begin{align*}
\|(\pi&\hatotimes \rho)([g_t,T\hatotimes  1_{\tilde D}])(\xi)\|_{L^2(G_x,r^*E)\hatotimes  H'}^2=
\\&=\|[(\tau\hatotimes \rho)(g_t),\,(\pi\hatotimes \rho)(T\hatotimes 1_{\tilde D})](\xi)\|_{L^2(G_x,r^*E)\hatotimes  H'}^2=
\\&=\int_{\gamma\in G_x}\left\|(\id_{E_{r(\gamma)}}\hatotimes \rho(g_t(r(\gamma)))\int_{\gamma_1\gamma_2=\gamma}(a(\gamma_1)\hatotimes  \id_{H'}) \xi(\gamma_2)-\right.
\\&\phantom{=\int_{\gamma\in G_x}\left\|\vphantom{\int}\right.}
\left.-\int_{\gamma_1\gamma_2=\gamma}(a(\gamma_1)\hatotimes  \id_{H'}) (\id_{E_{r(\gamma_2)}}\hatotimes \rho(g_t(r(\gamma_2)))  \xi(\gamma_2)\right\|_{E_{r(\gamma)}\hatotimes  H'}^2
\\&=\int_{\gamma\in G_x}\left\|\int_{\gamma_1\gamma_2=\gamma}(a(\gamma_1)\hatotimes  \rho(r^*g_t-s^*g_t)(\gamma_1)) \xi(\gamma_2)\right\|_{E_{r(\gamma)}\hatotimes  H'}^2
\\&\leq\int_{\gamma\in G_x}\left(\int_{\gamma_1\gamma_2=\gamma}\|(r^*g_t-s^*g_t)(\gamma_1)\|_{D}\cdot \|a(\gamma_1)\|_{(r^*E\hatotimes  s^*E^*)_{\gamma_1}}
\cdot \|\xi(\gamma_2)\|_{E_{r(\gamma_2)}\hatotimes  H'}\right)^2
\\&\leq\varepsilon_t^2\cdot \int_{\gamma\in G_x}\left(\int_{\gamma_1\gamma_2=\gamma}\|a(\gamma_1)\|_{(r^*E\hatotimes  s^*E^*)_{\gamma_1}}\cdot\|\xi(\gamma_2)\|_{E_{r(\gamma_2)}\hatotimes  H'}\right)^2
\\&=\varepsilon_t^2\cdot\left\|\pi_x(\!\left\bracevert\!a\!\right\bracevert\!)^2\left(\|\xi\|_{r^*E\hatotimes  H'}\right)\right\|^2\leq \varepsilon_t^2\cdot\|\pi_x(\!\left\bracevert\!a\!\right\bracevert\!)\|^2\cdot\|\xi\|_{L^2(G_x,r^*E)\hatotimes  H'}^2
\end{align*}
Here, $\!\left\bracevert\!a\!\right\bracevert\!\in C_c(G)$ denotes the point-wise norm of $a\in C_c(G,r^*E\hatotimes  s^*E^*)$.
It is a function in $C_c(G)$, because we have assumed that $a$ is supported in some coordinate chart and is therefore continuous in the usual sense.
Thus, the inequality
\[\|(\pi_x\hatotimes \rho)([g_t,T\hatotimes  1_{\tilde D}])\|\leq \varepsilon_t\cdot\|\pi_x(\!\left\bracevert\!a\!\right\bracevert\!)\|\]
holds for all $x\in M$.
Taking the supremum over all $x\in M$, we see that the norm of the commutator 
 \[[g_t,T\hatotimes 1_{\tilde D}]\in \K(\cE)\hatotimes_{\min} D\subset\mathfrak{B}\left(\bigoplus_{x\in M}L^2(G_x,r^*E)\hatotimes  H'\right)\]
is bounded  by $\varepsilon_t$ times the norm of $\!\left\bracevert\!a\!\right\bracevert\!\in \FA$ and thus tends to zero for $t\to\infty$. This implies $[\beta(g),\alpha(f)]=0$ in $\mathfrak{A}(\B(\cE)\hatotimes_{\min}\tilde D)$.
\end{proof}
\begin{question}
This proof uses an explicit calculation for the minimal tensor product. Nevertheless, one can ask whether the analogous statement still holds for the commutator of $\alpha_{\max},\beta_{\max}$.
\end{question}
According to this lemma, $\alpha$ and $\beta$ combine to a $*$-homomorphism
\[\mathfrak{m}_D:\,\kfrak(\FC,D)\hatotimes  \K(\cE)\to\mathfrak{A}( \B(\cE)\hatotimes _{\min}\tilde D).\]
Its image is contained in $\mathfrak{A}(\K(\cE)\hatotimes _{\min} D)$ and indeed it maps elementary tensors $\bar g\hatotimes  T$ as claimed. 
\end{proof}

From now on we will always assume that the coefficient algebras are nuclear to avoid mixing up minimal and maximal tensor products. \begin{thm}\label{thm:associativityproof}
Let $D_1,D_2$ be nuclear $C^*$-algebras. Then, in the asymptotic category,
\[(\llbracket \mathfrak{m}_{D_1}\rrbracket\hatotimes  1_{D_2})\circ(1_{\kfrak(\FC,D_1)}\hatotimes \llbracket \mathfrak{m}_{D_2}\rrbracket)
=\llbracket\mathfrak{m}_{D_1\hatotimes  D_2}\rrbracket\circ(\llbracket\nabla\rrbracket\hatotimes  1_{\K(\cE)}).\]
\end{thm}

Recall that the $*$-homomorphism
\[\nabla:\,\kfrak(\FC,D_1)\hatotimes \kfrak(\FC,D_2)\to\kfrak(\FC,D_1\hatotimes  D_2)\]
is defined by multiplication of functions.

\begin{proof}
The left hand side is represented by the $2$-homotopy class of
\begin{align*}
\kfrak(\FC,D_1)&\hatotimes \kfrak(\FC,D_2)\hatotimes  \K(\cE)\to\mathfrak{A}^2(\K(\cE)\hatotimes  D_1\hatotimes  D_2)
\\\bar f\hatotimes \bar g\hatotimes  T&\mapsto\overline{t\mapsto \overline{s\mapsto   (f_s\hatotimes  g_t)\cdot (T\hatotimes  1_{\tilde D_1}\hatotimes 1_{\tilde D_2}) }},
\end{align*}
while the right hand side is represented by the $1$-homotopy class of
\begin{align*}
\kfrak(\FC,D_1)&\hatotimes \kfrak(\FC,D_2)\hatotimes  \K(\cE)\to\mathfrak{A}(\K(\cE)\hatotimes  D_1\hatotimes  D_2)
\\\bar f\hatotimes \bar g\hatotimes  T&\mapsto \overline{s\mapsto (f_s\hatotimes  g_s)\cdot (T\hatotimes  1_{\tilde D_1}\hatotimes  1_{\tilde D_2})}.
\end{align*}
In these formulas, we have, of course, interpreted $f_s\hatotimes  g_t$ and $f_s\hatotimes  g_s$ as elements of 
$C(M)\hatotimes  D_1\hatotimes  D_2\subset\B(\cE)\hatotimes D_1\hatotimes  D_2$.

Similar to the definition of $\mathfrak{m}$, we construct a  2-homotopy 
\[
\kfrak(\FC,D_1)\hatotimes \kfrak(\FC,D_2)\hatotimes  \K(\cE)\to\mathfrak{A}^2((\K(\cE)\hatotimes  D_1\hatotimes  D_2)[0,1]).\]
There are three $*$-homomorphisms:
\begin{align*}
\tilde\alpha:\,\K(\cE)&\to \mathfrak{A}^2((\B(\cE)\hatotimes \tilde D_1\hatotimes \tilde D_2)[0,1])
\\T&\mapsto \overline{t\mapsto\overline{s\mapsto [r\mapsto T\hatotimes 1_{\tilde D_1}\hatotimes  1_{\tilde D_2}]  }}
\\\tilde\beta:\,\kbar(\FC,D_2)&\to \mathfrak{A}^2((\B(\cE)\hatotimes \tilde D_1\hatotimes \tilde D_2)[0,1])
\\g&\mapsto\overline{t\mapsto\overline{s\mapsto [r\mapsto   g_{rs+(1-r)t}\hatotimes 1_{\tilde D_1}]}}
\\\tilde\gamma:\,\kbar(\FC,D_1)&\to \mathfrak{A}^2((\B(\cE)\hatotimes \tilde D_1\hatotimes \tilde D_2)[0,1])
\\f&\mapsto\overline{t\mapsto\overline{s\mapsto [r\mapsto f_s\hatotimes 1_{\tilde D_2}   ]  }}
\end{align*}
There is only one non-trivial property to be verified here, namely the continuity of
\begin{align*}
\phi:\,[1,\infty)&\to \mathfrak{A}((\B(\cE)\hatotimes \tilde D_1\hatotimes \tilde D_2)[0,1])
\\t&\mapsto\overline{s\mapsto [r\mapsto   g_{rs+(1-r)t}\hatotimes 1_{\tilde D_1}   ]  }
\end{align*}
in the definition of $\tilde\beta$. 
This is a consequence of the following Lemma.
\begin{lem}
Let $X$ and $Y$ be metric spaces, $X$ complete. If $g\in C_b(X,Y)$ has vanishing variation, then it is uniformly continuous.
\end{lem}
\begin{proof}
Let $\varepsilon>0$. Because of vanishing variation there is a compact subset $K\subset X$ such that
\[(x\notin K\vee y\notin K)\wedge d(x,y)<1\quad\Rightarrow \quad d(g(x),g(y))<\varepsilon.\]
On the compact set $K$ however,  uniform continuity is automatic. 
Combining these features, the claim follows.
\end{proof}
Using this lemma and the distance estimate  
\[d(\,(rs+(1-r)t,x)\,,\,(rs+(1-r)t',x)\,)\leq |t-t'|\]
 in $\FC$, we conclude 
\begin{align*}
\|\phi(t)-\phi(t')\|&=\left\|  \overline{s\mapsto [r\mapsto   ( g_{rs+(1-r)t}- g_{rs+(1-r)t'}) \hatotimes  1_{\tilde D_2}  ]  }  \right\|
\\&\leq\sup_{s\in [0,\infty),r\in[0,1],x\in M}\left\|g(rs+(1-r)t,x) -  g(rs+(1-r)t',x) \right\|
\\&\xrightarrow{t'\to t}0.
\end{align*}

Now, $\tilde\beta$ and $\tilde\gamma$ factor through $\kfrak(\FC,D_2)$ and $\kfrak(\FC,D_1)$, respectively, and  their images commute with each other.
Furthermore, their images commute with the image of $\tilde\alpha$: The vanishing of $[\tilde\gamma(f),\tilde\alpha(T)]$ is completely analogous to Lemma \ref{asymptoticcommutation}. For the vanishing of $[\tilde\beta(g),\tilde\alpha(T)]$ we assume that $T\in \K(\cE)$ acts on $C_c(G,r^*E)$ by convolution with an $a\in C_c(G,r^*E\hatotimes  s^*E^*)$ compactly supported in some coordinate chart of $G$ and calculate
\begin{align*}
\|[\tilde\beta(g),\tilde\alpha(T)]\|
&=\limsup_{t\to\infty}\limsup_{s\to\infty}\sup_{r\in[0,1]}\|\,[ g_{rs+(1-r)t},T\hatotimes  1_{\tilde D_2}]\,\|
\\&\leq \limsup_{t\to\infty}\limsup_{s\to\infty}\sup_{r\in[0,1]}\varepsilon_{rs+(1-r)t}\cdot\|\!\left\bracevert\!a\!\right\bracevert\!\|_r=0,
\end{align*}
where $\varepsilon_{t}$, $\!\left\bracevert\!a\!\right\bracevert\!\in C_c(G)$ and the inequality are analogous to the ones in the proof of Lemma \ref{asymptoticcommutation}.

Thus, $\tilde\alpha$, $\tilde\beta$, $\tilde\gamma$ combine to a $2$-homotopy 
\[\kfrak(\FC,D_1)\hatotimes \kfrak(\FC,D_2)\hatotimes  \K(\cE)\to\mathfrak{A}^2((\B(\cE)\hatotimes \tilde D_1\hatotimes \tilde D_2)[0,1]).\]
The image is obviously contained in $\mathfrak{A}^2((\K(\cE)\hatotimes  D_1\hatotimes  D_2)[0,1])$ so it is actually  a  2-homotopy
\[\kfrak(\FC,D_1)\hatotimes \kfrak(\FC,D_2)\hatotimes  \K(\cE)\to\mathfrak{A}^2((\K(\cE)\hatotimes  D_1\hatotimes  D_2)[0,1]).\]
Evaluation at $0$ yields the representative of 
\[(\llbracket \mathfrak{m}_{D_1}\rrbracket\hatotimes  1_{D_2})\circ(1_{\kfrak(\FC,D_1)}\hatotimes \llbracket \mathfrak{m}_{D_2}\rrbracket)\]
while evaluation at $1$ is equivalent to the representative of 
\[\llbracket\mathfrak{m}_{D_1\hatotimes  D_2}\rrbracket\circ([\nabla]\hatotimes  1_{\K(\cE)}).\]
\end{proof}

One of our main results is now an easy corollary:
\begin{cor}[(cf.\ {\cite[Conjecture 0.2]{RoeFoliations}})]\label{cor:modulemultiplication}
The group $K_{*}(\K(\cE))$ is a module over $K_{FJ}^{-*}(M/\mathcal{F})$. The multiplication is 
\[K_{FJ}^{-i}(M/\mathcal{F})\otimes K_{j}(\K(\cE))\to K_{i+j}(\cfrak(\FC)\hatotimes \K(\cE))\xrightarrow{\llbracket\mathfrak{m}_{\K}\rrbracket\circ}K_{i+j}(\K(\cE)).\]
In particular, $K^*_C(M/\mathcal{F})$ is a module over $K_{FJ}^*(M/\mathcal{F})$.
\end{cor}
\begin{proof}
Associativity of the module multiplication is a direct consequence of Theorem \ref{thm:associativityproof} applied with $D_1=D_2=\K$.
If $E=\C\times M$ is the trivial one dimensional bundle, then $\K(\cE)=\FA$. This implies the special case mentioned, as $K^*_C(M/\mathcal{F})=K_{-*}(\FA)$ by definition.
\end{proof}

\begin{lem}
The Morita equivalence isomorphism
\[K_*(\K(\cE))\cong K_*(\FA)=K^{-*}_C(M/\mathcal{F})\]
provided by Corollary \ref{cor:bundlemorita}
is a module isomorphism.
\end{lem}
\begin{proof}
The inclusion $\K(\cE)\subset \MA_{I,J}(\C)\hatotimes \FA$, which induces this isomorphism in $K$-theory, obviously commutes with the asymptotic morphisms $\mathfrak{m}_{\K}$ associated to $\K(\cE)$ and $\MA_{I,J}(\C)\hatotimes \FA$.
\end{proof}

More general multiplications involving arbitrary coefficient $C^*$-algebras can be derived  similarly from Theorem \ref{thm:associativityproof}.

\section{Longitudinal index theory}\label{sec:longindtheory}
This section is a very short introduction to longitudinal index theory. For more details we refer primarily to \cite[Section 8.2]{Kordyukov}, but of course also to \cite{ConSurvey,ConNCG,ConSka}.

\begin{mdef}[({cf.\ \cite[Section 2.9]{ConNCG}})]
Let $\fol$ be a foliated mani\-fold, $E\to M$ a $\Z_2$-graded smooth vector bundle and $D:C^\infty(M,E)\to C^\infty(M,E)$ a  first order symmetric differential operator of grading degree one. The operator $D$ is called \emph{longitudinally elliptic} if it restricts to the leaves $L_x$ of the foliation and the restricted operators 
\[D_{L_x}:C_{c}^\infty(L_x,E|_{L_x})\to C_{c}^\infty(L_x,E|_{L_x})\]
are elliptic.
\end{mdef}

A special case are longitudinal Dirac type operators. Assume that $M$ is equipped with a Riemannian metric and  $E\to M$ is a $\Z_2$-graded, smooth, hermitian vector bundle equipped with a Clifford action of $T\mathcal{F}$ and a compatible connection $\nabla$. The associated Dirac operator is defined locally by the usual formula
\[D=\sum_{i=1}^{\dim\mathcal{F}}e_i\nabla_{e_i},\]
where $e_1,\dots,e_{\dim\mathcal{F}}$ is any local orthonormal frame of $T\mathcal{F}$.

If the bundle $T\mathcal{F}$ carries a $\operatorname{spin}^c$-structure, we obtain the longitudinal $\operatorname{spin}^c$-Dirac operator $\slashed{D}$ by choosing $E$ to be the corresponding spinor bundle. If the foliation is in addition even dimensional, then $E$ is $\Z_2$-graded and $\slashed{D}$ is a symmetric, grading degree one, longitudinally elliptic operator.

Let $\cE$ be the Hilbert module associated to $E$ and let \[D_{G_x}:C_c^\infty(G_x,r^*E)\to C_c^\infty(G_x,r^*E)\]
 be the lift of $D_{L_x}$ to the holonomy cover $G_x\to L_x$. 
The family  $\{D_{G_x}\,|\,x\in M\}$ assembles to a differential operator
\[D_G:\,
C_c^\infty(G,r^*E)\to C_c^\infty(G,r^*E).\]
Its closure, which we also denote by $D_G$, is an odd, unbounded, selfadjoint operator on the Hilbert module $\mathcal{E}$ constructed in Section \ref{sec:modules}. It is regular in the sense of \cite{BaaJul}. The proof of regularity relies on the existence of a parametrix and can be found in \cite[Proposition 3.4.9]{Vassout}.
\begin{mdef}[({\cite[Section 8.2]{Kordyukov}})]
The $KK$-theory class of  $D$ is the element 
\[[D]\in KK(C(M),\FA)\] given in the unbounded picture of $KK$-theory by the 
triple $(\cE,\phi,D_G)$ where $\phi:C(M)\to\B(\cE)$ is the inclusion of Lemma \ref{lem:subofmultiplier}. The index of $D$ is the element
\[\ind(D)\in K(\FA)=K_C^0(M/\mathcal{F})\]
obtained from $[D]$ by crushing $M$ in the first variable to a point.
\end{mdef}

There are several advantages of having the $KK$-theory class $[D]$ at hand. One of them is the usual index pairing with vector bundles:
Given a longitudinally elliptic operator $D$ and a smooth vector bundle $F\to M$, we can construct the twisted operator
\[D_F=D\otimes F:C_c^\infty(M,E\otimes F)\to C_c^\infty(M,E\otimes F),\]
which is again a longitudinally elliptic operator.
A direct consequence of \cite[Theorem 13]{Kucerovsky} is the following generalization of the index pairing formula.
\begin{lem}\label{TwistedIndexFormulaviaKKclass}
$\ind(D_F)=[D]\circ[F]\in K_C^0(M/\mathcal{F})$.
\end{lem}

Another appearence of the $KK$-theory class is the following. If $T\mathcal{F}$ is even dimensional and endowed with a $\operatorname{spin}^c$ structure, then the map $p:M\to M/\mathcal{F}$ of Example \ref{ex:projtoleafspace} is $K$-oriented and induces a wrong way map $p!:K^*(M)\to K^*_C(M/\mathcal{F})$ given by composition product with an element $p!\in KK(C(M),\FA)$ \cite{ConSka}. 
In this particular case, $p!$ is in fact the $KK$-theory class of the $\operatorname{spin}^c$-Dirac operator $\slashed{D}$.

For our purposes, we have to pass from $KK$- to $E$-theory. Under the canonical isomorphism of Proposition \ref{prop:KKEiso}, $[D]$ corresponds to the following $E$-theory classes:
\begin{mdef}\label{def:Etheorylongfundclass}
The $E$-theory class 
\[\llbracket D\rrbracket\in E(C(M),\K(\cE))\cong E(C(M),\FA)\]
 of  a longitudinally elliptic Dirac type operator $D$ over $\fol$ is represented by the asymptotic morphism 
\[\rho:\,\mathcal{S}\otimes C(M)\to\mathfrak{A}(\K(\cE)),\quad f\otimes g\mapsto\overline{t\mapsto g\cdot f(t^{-1}D_G)}.\]
\end{mdef}
\begin{lem}\label{lem:indexnonasymptoticformula}
The index 
\[\ind(D)\in E(\C,\K(\cE))\cong E(\C,\FA)\cong K_C^0(M/\mathcal{F})\]
is represented by the $*$-homomorphism 
\begin{equation}\label{eq:indexnonasymptoticformula}
\mathcal{S}\to\K(\cE),\quad f\mapsto  f(D_G).
\end{equation}
\end{lem}
\begin{proof}
By definition, $\ind(D)$ is represented by  the asymptotic morphism 
\begin{equation}\label{eq:indexasymptoticformula}
\mathcal{S}\to\mathfrak{A}(\K(\cE)),\quad f\mapsto \overline{t\mapsto f(t^{-1}D_G)}\,.
\end{equation}
A natural candidate for a $1$-homotopy between \eqref{eq:indexasymptoticformula} and \eqref{eq:indexnonasymptoticformula} is 
\begin{equation}\label{eq:indexonehomotopy}
\mathcal{S}\to\mathfrak{A}(\K(\cE)[0,1]),\quad f\mapsto \overline{t\mapsto[r\mapsto f((r+(1-r)t^{-1})D_G)]}\,.
\end{equation}
We have to show that, for each $f\in \mathcal{S}$, the function
\[[1,\infty)\times[0,1]\to \K(\cE),\quad (t,r)\mapsto f((r+(1-r)t^{-1})D_G)\]
is continuous. This follows from continuity of 
\[(0,1]\to\mathcal{S},\quad \lambda\mapsto f(\lambda\cdot\,\_\,)\]
and the continuity of the functional calculus
\[\mathcal{S}\to\K(\cE),\quad f\mapsto f(D_G).\]
Furthermore, \eqref{eq:indexonehomotopy} is a $*$-homomorphism, because the functional calculus is. 
Thus, it fulfills all requirements on a $1$-homotopy.
\end{proof}

\section{Twisted operators and the module structure}\label{sec:indexcalculations}
In this section, we clarify the relation of our module structure to index theory. To this end, we make the following definition:
\begin{mdef}
We denote by $\llbracket p^*\rrbracket\in E(\cfrak(\FC),C(M))$ the $E$-theory class of the asymptotic morphism
\[p^*:\cfrak(\FC)\to \mathfrak{A}(C(M)\otimes\K),\quad\bar g\mapsto\overline{t\mapsto g_t}.\]
\end{mdef}
This notation comes from the fact that the composition product with $\llbracket p^*\rrbracket$ yields the homomorphism
\[p^*:K^*_{FJ}(M/\mathcal{F})\to K^*_{FJ}(M)\cong K^*(M)\]
induced by the smooth map of leaf spaces  $p:M\to M/\mathcal{F}$ defined in Example \ref{ex:projtoleafspace}.
To see this, recall that the isomorphism $K^*_{FJ}(M)\cong K^*(M)$ comes from the  inclusion $C(M)\otimes\K\subset\cfrak(\mathcal{O}M)$ as constant functions. An inverse to this isomorphism is induced by the asymptotic morphism
\[\cfrak(\mathcal{O}M)\to \mathfrak{A}(C(M)\otimes\K),\quad\bar g\mapsto\overline{t\mapsto g_t},\]
because the composition
\[C(M)\otimes\K\to \cfrak(\mathcal{O}M)\to\mathfrak{A}(C(M)\otimes\K)\]
is simply the inclusion as constant functions and therefore the identity morphism in the asymptotic category.
The claim now follows from the fact that the homomorphism $p^*:K^*_{FJ}(M/\mathcal{F})\to K^*_{FJ}(M)$ comes from  the inclusion
$\cfrak(\FC)\subset\cfrak(\mathcal{O}M)$.

Our main result relating the module structure to index theory is the following:
\begin{thm}\label{thm:superformula}
Let $D$ be a longitudinally elliptic operator over $\fol$ and $\llbracket p^*\rrbracket\in E(\cfrak(\FC),C(M))$ the element defined above. Then
\[\llbracket D\rrbracket\circ\llbracket p^*\rrbracket=\llbracket\mathfrak{m}_\K\rrbracket\circ(1_{\cfrak(\FC)}\otimes\ind(D))\in E(\cfrak(\FC),\K(\cE)).\]
\end{thm}
Before proving this theorem, here are two consequences:
\begin{cor}\label{cor:twistingbyasymptoticallyleafwisebundles}
If $D$ is a longitudinally elliptic  operator, $F\to M$ a smooth vector bundle for which there is an element $x_F\in K^0_{FJ}(M/\mathcal{F})$ with $[F]=p^*(x_F)$ (e.\,g.\ $F$ asymptotically a bundle over the leaf space as in Definition \ref{asymptoticallyleafwisebundles}), 
then the index of the twisted operator $D_F$ is
\[\ind(D_F)=x_F\cdot \ind(D)\in K^0_C(M/\mathcal{F}).\]
\end{cor}
\begin{proof}
$\ind(D_F)=[D]\circ[F]=[D]\circ\llbracket p^*\rrbracket\circ x_F=\llbracket\mathfrak{m}_{\K}\rrbracket \circ(x_F\otimes\ind(D))=x_F\cdot\ind(D).$
\end{proof}
\begin{cor}[(cf.\ {\cite[p.\ 204]{RoeFoliations}})]\label{cor:shriekmap}
Assume that $T\mathcal{F}$ is even dimensional and $\operatorname{spin}^c$ and let $\slashed{D}$ be the corresponding Dirac operator.
Then the map
\[p_!\circ p^*:K^*_{FJ}(M/\mathcal{F})\to K^*_C(M/\mathcal{F})\]
is module multiplication with $\ind(\slashed{D})\in K^0_C(M/\mathcal{F})$.
\end{cor}
\begin{proof}
$p_!\circ p^*(x)=\llbracket \slashed{D}\rrbracket\circ\llbracket p^*\rrbracket\circ x=\llbracket\mathfrak{m}_{\K}\rrbracket \circ(x\otimes\ind(\slashed{D}))=x\cdot\ind(\slashed{D}).$
\end{proof}

For the proof of Theorem \ref{thm:superformula} it would be beneficial if one had defined the stable Higson corona in a more analytic way. Recall that Higson originally defined his corona $\eta X$ of a complete Riemannian manifold $X$ as the maximal ideal space of the  $C^*$-algebra generated by the bounded  smooth functions $X\to\C$ whose gradient vanishes at infinity (cf.\ \cite[Section 5.1]{RoeCCITCRM}).
It is unknown to the author under which conditions this definition is equivalent to Roe's definition (ibid.). In other words: given a complete Riemannian manifold and a bounded continuous  function of vanishing variation, can this function be approximated by  smooth functions whose gradients vanish at infinity? For the present situation, it is sufficient to have the following partial result for the stable Higson corona of foliated cones:

\begin{lem}\label{lem:smoothapprox}
Every element of $\cfrak(\FC)$ has a representative $g\in \cbar(\FC)$ such that $g_t\in C(M)\otimes\K$ is differentiable in the leafwise direction for all $t$ and  the leafwise derivatives $X.g_t\in C(M)\otimes \K$ vanish in the limit $t\to\infty$ for every leafwise vector field $X\in C(M,T\mathcal{F})$.
\end{lem}

\begin{proof}
Let $\{\phi_i\}_{i=1,\dots,k}$ be an atlas of foliation charts $\phi_i:U_i\xrightarrow{\approx}\R^{\dim\mathcal{F}}\times\R^{\codim\mathcal{F}}$ and $\{\chi_i\}_{i=1,\dots,k}$ a subordinate smooth partition of unity. Choose a smooth function
$\delta:\R^{\dim\mathcal{F}}\to[0,\infty)$ supported in the compact unit ball $\overline{B}_1(0)$ such that $\int\delta=1$.

Given  any $h\in \cbar(\FC)$, we define the functions 
\[g_i,h_i:[0,\infty)\times \R^{\dim\mathcal{F}}\times\R^{\codim\mathcal{F}}\to \K\]
for $i=1,\dots,k$ by the formulas
\begin{align*}
h_i(t,x,z)&:=h(t,\phi_i^{-1}(x,z)),
\\g_i(t,x,z)&:=\int_{\R^{\dim\mathcal{F}}}\delta(x-y)h_i(t,y,z)dy
\end{align*}
and $g:\FC\to\K$ by 
\[g(t,p):=\sum_{i=1}^k\chi_i(p)g_i(t,\phi_i(p)).\]
This function $g$ is clearly continuous and we claim that it is a representative of $\overline{h}\in \cfrak(\FC)$ with the desired properties.

To this end, let $K_i:=\supp(\chi_i)$ and  note that there is $R>0$ with the following property: 
Whenever $i=1,\dots,k$ and  $(x,z),(y,z)\in \phi_i(K_i)+\overline{B}_1(0)$, then the points
$\phi_i^{-1}(x,z),\phi_i^{-1}(y,z)\in M$ are joined by a leafwise path of length  at most $R$. In particular, for all $t\geq 0$ the distance between the two points
\[(t,\phi_i^{-1}(x,z)),\,(t,\phi_i^{-1}(y,z))\in \FC\]
is at most $R$ and therefore
\[\|h_i(t,x,z)-h_i(t,y,z)\|\leq \operatorname{Var}_Rh(t,\phi_i^{-1}(x,z)).\]
This implies 
\begin{align*}
\|g_i(t,x,z)-h_i(t,x,z)\|&=\left\|\int_{\overline{B}_1(x)}\delta(x-y)(h_i(t,y,z)-h_i(t,x,z))dy\right\|
\\&\leq \int_{\overline{B}_1(x)}\delta(x-y)\|h_i(t,y,z)-h_i(t,x,z)\|dy
\\&\leq\operatorname{Var}_Rh(t,\phi_i^{-1}(x,z))
\end{align*}
for all $t\geq 0$ and $(x,z)\in \phi_i(K_i)$ and therefore
$\|g(t,p)-h(t,p)\|\leq \operatorname{Var}_Rh(t,p)$ for all $(t,p)\in \FC$.
As $\operatorname{Var}_Rh$ vanishes at infinity, $g$ must also have vanishing variation and $\overline{g}=\overline{h}$ in $\cfrak(\FC)$.

It remains to estimate the longitudinal derivatives: Let $X\in C(M,T\mathcal{F})$ be a leafwise vector field. 
Given $p\in M$, we denote $(x_i,z_i):=\phi_i(p)$ and 
\[X_i(x_i,z_i):=d\phi_i(X(p))\in \R^{\dim\mathcal{F}}\subset \R^{\dim\mathcal{F}}\times\R^{\codim\mathcal{F}}\]
for those $i$ with $p\in U_i$. The derivative of $g_t$ along the leafwise vector field $X$ is
\[X.g_t(p)=\sum_{i=1}^kX.\chi_i(p)g_i(t,x_i,z_i)+\sum_{i=1}^k\chi_i(p)\int_{\overline{B}_1(x_i)}X_i.\delta(x_i-y)h_i(t,y,z_i)dy.\]

As for the first summand, note that 
\[\sum_{i=1}^kX.\chi_i(p)h_i(t,x_i,z_i)=\sum_{i=1}^kX.\chi_i(p)h(t,p)=(X.1)h(t,p)=0\]
and therefore
\begin{align*}
\left\|\sum_{i=1}^kX.\chi_i(p)g_i(t,x_i,z_i)\right\|&=\left\|\sum_{i=1}^kX.\chi_i(p)(g_i(t,x_i,z_i)-h_i(t,x_i,z_i))\right\|
\\&\leq\sum_{i=1}^k|X.\chi_i(p)|\operatorname{Var}_Rh(t,p)
\end{align*}
which converges to $0$ uniformly in $p\in M$ for $t\to\infty$.

For the second summand, we use that $\int\delta\equiv 1$ and therefore
\[\int_{\overline{B}_1(x_i)}X_i.\delta(x_i-y)dy=0\]
to estimate
\begin{align*}
\left\|\int_{\overline{B}_1(x_i)}X_i.\delta(x_i-y)h_i(t,y,z_i)dy\right\|=
&\left\|\int_{\overline{B}_1(x_i)}X_i.\delta(x_i-y)(h_i(t,y,z_i)-h_i(t,x_i,z_i))dy\right\|
\\&\leq\int_{\overline{B}_1(x_i)}|X_i.\delta(x_i-y)|\operatorname{Var}_Rh(t,p)dy
\end{align*}
which also converges to $0$ uniformly in $p\in M$ for $t\to\infty$, because $X_i.\delta$ is bounded on the compact set $K_i+\overline{B}_1(0)$.
\end{proof}

\begin{proof}[of Theorem \ref{thm:superformula}]
The left hand side is represented by
\begin{align*}
\mathcal{S}\hatotimes  \cfrak(\FC)&\to\mathfrak{A}^2(\K(\cE)\hatotimes \K)
\\f\hatotimes  g&\mapsto \overline{t\mapsto\overline{s\mapsto g_t\cdot (f(s^{-1}D_G)\hatotimes  1_{\tilde\K})}}\,.
\end{align*}
The right hand side, on the other hand, is represented by 
\begin{align*}
\mathcal{S}\hatotimes  \cfrak(\FC)&\to\mathfrak{A}(\K(\cE)\hatotimes \K)
\\f\hatotimes  g&\mapsto \overline{t\mapsto g_t\cdot (f(D_G)\hatotimes  1_{\tilde\K})}\,.
\end{align*}
To construct a $2$-homotopy between them, let
\begin{align*}
\alpha:\,\mathcal{S}&\to\mathfrak{A}^2((\B(\cE)\hatotimes \tilde\K)[0,1])
\\f&\mapsto\overline{t\mapsto \overline{s\mapsto[r\mapsto f((r+(1-r)s^{-1})D_G)\hatotimes 1_{\tilde \K}]}}
\end{align*}
be the $2$-homotopy obtained from the $1$-homotopy in the proof of  Lemma \ref{lem:indexnonasymptoticformula}. 
Furthermore, by interpreting $g_t\in C(M)\otimes\K$ as an element of $\B(\cE)\hatotimes \K$ for all $t$, we obtain an inclusion
\[\beta:\,\cfrak(\FC)\to\mathfrak{A}^2((\B(\cE)\hatotimes \tilde\K)[0,1]),\quad\bar g\mapsto \overline{t\mapsto \overline{s\mapsto[r\mapsto g_t]}}.\]
We have to show that the commutators $[\alpha(f),\beta(\bar g)]$ vanish for all $f\in \mathcal{S}$ and $\bar g\in \cfrak(\FC)$.
We may assume $f(x)=(x\pm i)^{-1}$, as these functions generate $\mathcal{S}$, and that $g$ is as in Lemma \ref{lem:smoothapprox}.
Then
\begin{align*}
\|[\alpha(f),\beta(\bar g)]\|&=\limsup_{t\to\infty}\limsup_{s\to\infty}\sup_{r\in[0,1]}\|[f((r+(1-r)s^{-1})D_G)\hatotimes 1_{\tilde \K},g_t]\|
\\&=\limsup_{t\to\infty}\sup_{\lambda\in(0,1]}\|[f(\lambda D_G)\hatotimes 1_{\tilde \K},g_t]\|
\\&=\limsup_{t\to\infty}\sup_{\lambda\in(0,1]}\| \lambda\cdot((\lambda D_G\pm i)^{-1}\hatotimes  1_{\tilde\K})\cdot[D_G\hatotimes 1_{\tilde\K},g_t]\cdot
((\lambda D_G\pm i)^{-1}\hatotimes  1_{\tilde\K})\|
\\&\leq \limsup_{t\to\infty}\|[D_G\hatotimes 1_{\tilde\K},g_t]\|.
\end{align*}
If we write $D=\sum_iA_iX_i$ with bundle endomorphisms $A_i\in C(M, \operatorname{End}(E))$ and leafwise vector fields $X_i\in  C(M,T\mathcal{F})$, then
\[[D_G\hatotimes 1_{\tilde\K},g_t]=\sum_i(A_i\hatotimes  1_{\tilde\K})(1_{\operatorname{End}(E)}\hatotimes  X_i.g_t)\in C(M,\operatorname{End}(E))\hatotimes \K\]
and this vanishes for $t\to\infty$ by the choice of $g$.

Thus, $\alpha$ and $\beta$ combine to a $2$-homotopy
\begin{align*}
\mathcal{S}\hatotimes \cfrak(\FC)&\to\mathfrak{A}^2(C[0,1]\hatotimes \B(\cE)\hatotimes \tilde\K)
\\f\hatotimes  g&\mapsto \overline{t\mapsto \overline{s\mapsto[r\mapsto g_t\cdot (f((r+(1-r)s^{-1})D_G)\hatotimes  1_{\tilde\K})]}}
\end{align*}
whose image lies in the sub-$C^*$-algebra $\mathfrak{A}^2(C[0,1]\hatotimes \K(\cE)\hatotimes \K)$.
Evaluating at $0,1$ yields representatives of left and right hand side of the equation.
\end{proof}

%\bibliographystyle{alpha}
%\bibliography{diss.bib}

\affiliationone{
Christopher Wulff\\
Instituto de Matem\'aticas (Unidad Cuernavaca), Universidad Nacional Aut\'onoma de M\'exico,
Avenida Universidad s/n, Colonia Lomas de Chamilpa, 62210 Cuernavaca, Morelos,\\
M\'exico \email{christopher.wulff@im.unam.mx}
}

\end{document}